\documentclass[12pt, a4paper]{article}

\usepackage{graphicx}

\usepackage{amssymb}
\usepackage{amsthm}

\usepackage{amsmath}
\usepackage{hyphenat}

\usepackage{examplep}

\usepackage[english]{babel}

\usepackage{booktabs}

\usepackage{subfig}

\usepackage{float}
\usepackage{xspace}
\usepackage{color}
\usepackage{xcolor}
\usepackage[hidelinks]{hyperref} 
\hypersetup{
	colorlinks}
\usepackage[font=scriptsize]{caption} 
\usepackage{subcaption}
\newcommand{\ii}{\operatorname{i}}
\newcommand{\ee}{\operatorname{e}}

\newcommand{\dd}{\mathrm{d}}
\newcommand{\Comma}{\: ,}
\newcommand{\period}{\: .}
\newcommand{\semicolon}{\: ;}
\newcommand{\Ai}{\operatorname{Ai}}
\newcommand{\Bi}{\operatorname{Bi}}
\newcommand{\MATLAB}{\textsc{Matlab}\xspace}

\newcommand{\anton}{\textcolor{black}}
\newcommand{\jannis}{\textcolor{black}}
\newcommand{\markus}{\textcolor{black}}

\newtheorem{Theorem}{Theorem}[section]
\newtheorem{Remark}[Theorem]{Remark}
\newtheorem{Lemma}[Theorem]{Lemma}
\newtheorem{Corollary}[Theorem]{Corollary}
\newtheorem{Proposition}[Theorem]{Proposition}
\newtheorem{innercustomthm}{Hypothesis}
\newenvironment{Hypothesis}[1]
  {\renewcommand\theinnercustomthm{#1}\innercustomthm}
  {\endinnercustomthm}

\title{\vspace{-2.5cm}Optimally truncated WKB approximation for the 1D stationary Schrödinger equation in the highly oscillatory regime}
 \author{
    Anton Arnold\footremember{label1}{Corresponding author.}\textsuperscript{,}\footremember{alley}{Institute of Analysis and Scientific              Computing, Technische Universität Wien, Wiedner Hauptstr. 8-10, A--1040 Wien, Austria,\\
    \href{mailto:anton.arnold@tuwien.ac.at}{anton.arnold@tuwien.ac.at}, \href{mailto:jannis.koerner@tuwien.ac.at}{jannis.koerner@tuwien.ac.at}, \href{mailto:melenk@tuwien.ac.at}{melenk@tuwien.ac.at}}
    \and
    Christian Klein\footremember{trailer}{Institut de Math\'{e}matiques de             Bourgogne, Institut Universitaire de France, Universit\'{e} de Bourgogne, 9 avenue Alain           Savary, \markus{F--21078 Dijon},  France, \href{mailto:Christian.Klein@u-bourgogne.fr}{Christian.Klein@u-bourgogne.fr}}
    \and
    Jannis Körner\footrecall{alley}
    \and
    Jens Markus Melenk\footrecall{alley} 
}

\date{\vspace{-2em}}

\begin{document}

\setcounter{page}{1} \thispagestyle{empty}

	

\newcommand{\footremember}[2]{%
    \footnote{#2}
    \newcounter{#1}
    \setcounter{#1}{\value{footnote}}%
}
\newcommand{\footrecall}[1]{%
    \footnotemark[\value{#1}]%
}

\maketitle

%
\begin{abstract}
	This paper is dedicated to the efficient numerical computation of solutions to the 1D stationary Schrödinger equation in the highly oscillatory regime. 
    We compute an approximate solution based on the well-known WKB-ansatz, which relies on an asymptotic expansion w.r.t.\ the small parameter $\varepsilon$. Assuming that the coefficient in the equation is analytic, we derive an explicit error estimate for the truncated WKB series, in terms of $\varepsilon$ and the truncation order $N$. For any fixed $\varepsilon$, this allows to determine the optimal truncation order $N_{opt}$ which turns out to be proportional to $\varepsilon^{-1}$. 
    When chosen this way, the resulting error of the \emph{optimally truncated WKB series} behaves like 
    \anton{$\mathcal{O}(\exp(-r/\varepsilon))$,  }
    with some parameter $r>0$. The theoretical results established in this paper are confirmed by several numerical examples.
\end{abstract}



%

\section{Introduction}
In this paper we are concerned with the numerical solution of the 
stationary 1D Schrödinger equation
\begin{align}\label{schroedinger_eq}
	\begin{cases}
		\varepsilon^2\varphi''(x)+a(x)\varphi(x)=0\Comma \quad x\in I:=[\xi,\eta]\Comma\\
		\varphi(\xi)=\varphi_{0}\Comma\\
		\varepsilon\varphi'(\xi)=\varphi_{1}\Comma
	\end{cases}	
\end{align}
\anton{which yields highly oscillatory solutions.}
Here, $0<\varepsilon\ll 1$ is a very small parameter and $a$ is a real-valued function satisfying $a(x)\geq a_{0} > 0$ and, for a quantum mechanical problem, it is related to the potential. The constants $\varphi_{0},\varphi_{1}\in\mathbb{C}$ may depend on $\varepsilon$ but are assumed to be ${\varepsilon}$-uniformly bounded. It is known that the (local) \anton{wavelength} $\lambda$ of the solution $\varphi$ to (\ref{schroedinger_eq}) is proportional to $\varepsilon$. More precisely, it can be expressed as $\lambda(x)=(2\pi\varepsilon)/\sqrt{a(x)}$. Consequently, for a small parameter $\varepsilon$ the solution becomes highly oscillatory, particularly in the semi-classical limit $\varepsilon \to 0$.

Highly oscillatory problems such as (\ref{schroedinger_eq}) occur across a broad range of applications, e.g., plasma physics \cite{Courant1958TheoryOT,Lewis1968MOTIONOA}, inflationary cosmology \cite{Martin2003WKBAF,Winitzki2005CosmologicalPP} and electron transport in semiconductor devices such as resonant tunneling diodes \cite{Mennemann2013TransientSS,Sun1998ResonantTD,Negulescu2005PHD}. 
More specifically, the state of an electron of mass $m$ that is injected with the prescribed energy $E$ from the right boundary into an electronic device (e.g., diode), modeled on the interval $[\xi, \eta]$, can be described by the following boundary value problem (BVP) (e.g., see \cite{Arnold2011WKBBasedSF} or \cite[Chap.\ 2]{Negulescu2005PHD}):
\begin{align}\label{schroedinger_BVP}
    \begin{cases}
		-\varepsilon^2\psi_{E}''(x)+V(x)\psi_{E}(x)=E\psi_{E}(x)\Comma \quad x\in (\xi,\eta)\Comma\\
		\psi_{E}'(\xi)+\ii k(\xi)\anton{\psi_{E}(\xi)}=0\Comma\\
		\psi_{E}'(\eta)-\ii k(\eta)\anton{\psi_{E}(\eta)}=-2\ii k(\eta)\period
	\end{cases}
\end{align}
Here, $\varepsilon:=\hbar/\sqrt{2m}$ is proportional to the (reduced) Planck constant $\hbar$, $k(x):=\varepsilon^{-1}\sqrt{E-V(x)}$ is the \anton{wavevector} and the real-valued function $V$ denotes the electrostatic potential. In the context of (\ref{schroedinger_BVP}), our assumption $a(x)\geq a_{0}>0$ simply reads $E>V(x)$, which means that we are in the oscillatory regime. One is then often interested in macroscopic quantities such as the charge density $n$ and the current density $j$, which are given by
\begin{align}\label{densities}
    n(x)=\int_{0}^{\infty}|\psi_{E}(x)|^{2}f(E)\,\mathrm{d}E\Comma \quad j(x)=\varepsilon\int_{0}^{\infty}\operatorname{Im}(\overline{\psi_{E}(x)}\psi_{E}'(x))f(E)\,\mathrm{d}E\period
\end{align}
Here, $f$ is the distribution function which represents the injection statistics of the electron and $\operatorname{Im}(\cdot)$ denotes the imaginary part. Thus, in order to compute the quantities (\ref{densities}), one has to use a very fine grid in $E$ which means that the BVP (\ref{schroedinger_BVP}) has to be solved many times. Consequently, there exists a substantial demand for efficient numerical methods that are suitable for solving problems like (\ref{schroedinger_BVP}). Further, we note that the BVP (\ref{schroedinger_BVP}) is strongly connected to IVP (\ref{schroedinger_eq}). Indeed, for suitable initial values, namely, $\varphi_{0}=1$ and $\varphi_{1}=-\ii\sqrt{a(\xi)}$, the solution $\varphi$ of IVP (\ref{schroedinger_eq}) and the solution $\psi_{E}$ of BVP (\ref{schroedinger_BVP}) are related by
\begin{align}
    \psi_{E}(x)=-\frac{2\ii k(\eta)}{\varphi'(\eta)-\ii k(\eta)\varphi(\eta)}\varphi(x)\period
\end{align}
Thus, any numerical method for solving IVP (\ref{schroedinger_eq}) is also suitable for the numerical treatment of BVP (\ref{schroedinger_BVP}).
\subsection{Background and approach}
Since the solution $\varphi$ to (\ref{schroedinger_eq}) exhibits rapid oscillations when $\varepsilon$ is small, standard numerical methods for ODEs become inefficient as they are typically constrained by grid limitations $h=\mathcal{O}(\varepsilon)$ ($h$ denoting the step size), in order to resolve the oscillations accurately. By contrast,
\anton{the \emph{phase function method} of \cite{Bremer2018} is based on the observation that solutions to \eqref{schroedinger_eq} can be represented accurately by means of a nonoscillatory phase function. Our approach presented below is closer to the \textit{uniformly accurate} (w.r.t.\ $\varepsilon$) marching methods of  \cite{Lorenz2005AdiabaticIF,Jahnke2003NumericalIF} } 
which yield global errors of order $\mathcal{O}(h^{2})$ and allow to reduce the grid limitation to at least $h=\mathcal{O}(\sqrt{\varepsilon})$.
The WKB-based (named after the physicists Wentzel, Kramers, Brillouin; cf. \cite{Landau1985Q}) one-step method from \cite{Arnold2011WKBBasedSF} is even \textit{asymptotically correct}, i.e.\ the numerical error goes to zero with $\varepsilon\to 0$, provided that the integrals $\int^{x}\sqrt{a(\tau)}\,\mathrm{d}\tau$ and $\int^{x}a(\tau)^{-1/4}(a(\tau)^{-1/4})''\,\mathrm{d}\tau$ for the phase of the solution can be computed exactly. More precisely, the method then yields an error which is of order $\mathcal{O}(\varepsilon^{3})$ as $\varepsilon\to 0$ and $\mathcal{O}(h^{2})$ as $h\to 0$. If these integrals cannot be evaluated exactly, the asymptotically correct error behavior can be (almost) recovered by employing spectral methods for the integrals, as shown in \cite{Arnold2022WKBmethodFT}.
Further, in \cite{Agocs2022AnAS} the authors propose a numerical algorithm that switches adaptively between a defect correction iteration (which builds on an asymptotic expansion) for oscillatory regions of the solution, and a conventional Chebyshev collocation solver for smoother regions. Although the method is demonstrated to be highly accurate and efficient, a full error analysis was left for future work.

Our approach here is to implement directly a WKB approximation for the solution of (\ref{schroedinger_eq}), which is asymptotically correct and of arbitrary order w.r.t.\ $\varepsilon$. The essence of the method is rather an analytic approximation via an asymptotic WKB series with optimal truncation. As such, the main goal is to understand the asymptotic $\varepsilon$-dependence of this truncation strategy and of the resulting error. Thus our strategy is not a classical numerical method with some chosen grid size $h$ and convergence as $h\to 0$. Instead, the resulting approximation error will be of order $\mathcal{O}(\varepsilon^{N})$ as $\varepsilon\to 0$, where $N$ refers to the used truncation order in the underlying asymptotic WKB series, see (\ref{WKB-ansatz})-(\ref{wkb_S}) below. As $N$ can be chosen freely, this approach may prove very practical for applications, especially when the model parameter $\varepsilon$ is very small. Since the computation of the terms of the asymptotic series involves several integrals, we will employ highly accurate spectral methods, as already proven useful in \cite{Arnold2022WKBmethodFT}.

The key question when implementing this WKB approximation is which choice of $N$ is adequate or even optimal, in the sense of minimizing the resulting approximation error. Indeed, since the asymptotic WKB series is typically divergent, the error cannot simply be reduced further by increasing the value of $N$. This question about the best attainable accuracy of the WKB approximation was already addressed in \cite{Winitzki2005CosmologicalPP}, where the author compared the WKB series with the exact solution represented by a convergent \anton{Bremmer series \cite{Bremmer1951, Atkinson1960JMAA},} or more precisely, by an asymptotic expansion of that Bremmer series. The author finds that in cases where the coefficient function $a$ is analytic, the optimal truncation order is proportional to $\varepsilon^{-1}$, yielding a corresponding optimal accuracy which is exponentially small w.r.t.\ $\varepsilon$. However, to derive these results, the author makes several additional asymptotic approximations.
In the present paper, on the other hand, we shall follow a more rigorous strategy by providing error estimates for the WKB approximation which are explicit w.r.t.\ $\varepsilon$ \textit{and} $N$. 
We note, however, that the key assumption from \cite{Winitzki2005CosmologicalPP}, i.e., $a$ being analytic, will also be crucial for the strategy of the present work. 

\anton{In practical finite precision computations, 
optimal truncation is not generally needed since it is not useful to add additional terms after reaching machine precision. In this paper we present concrete a-priori estimates for this truncation. }
\subsection{Main results}
%
Our first main result is Theorem~\ref{theorem_wkb_error}, which provides an explicit (w.r.t.\ $\varepsilon$ and $N$) error estimate for the WKB approximation, and implies that the approximation error is of order $\mathcal{O}(\varepsilon^{N})$. The explicitness of this estimate then allows the investigation of the error w.r.t.\ the truncation order $N$. Indeed, the optimal truncation order $N_{opt}$ can be predicted by minimizing the established upper error bound w.r.t.\ $N$ or by determining the smallest term of the asymptotic series, and is found to be proportional to $\varepsilon^{-1}$. This leads to our second main result, namely, Corollary~\ref{corollary_exp_small_error}. It states that, for an adequate choice of $N=N(\varepsilon)\sim \varepsilon^{-1}$, the error of the WKB approximation is of order 
\anton{$\mathcal{O}(\exp(-r/\varepsilon))$, }
$r>0$ being some constant. As a consequence, also the error of the optimally truncated WKB approximation is of order 
\anton{$\mathcal{O}(\exp(-r/\varepsilon))$, }
see also Remark~\ref{remark_optimal_error}.

This paper is organized as follows: In Section~\ref{section:wkb_approximation} we introduce the $N$-th order (w.r.t.\ $\varepsilon$) WKB approximation as an approximate solution of IVP (\ref{schroedinger_eq}). Section~\ref{section_error_analysis} then contains a detailed error analysis for the WKB approximation and includes explicit (w.r.t.\ $\varepsilon$ and the truncation order $N$) error estimates. In Section~\ref{section_computation_of_the_WKB_approximation} we specify the computation of the WKB approximation. This includes the description of the chosen methods for the computation of the terms of the underlying asymptotic series as well as a reasonable truncation strategy. In Section~\ref{section_numerical_simulations} we illustrate the theoretical results established in this paper by several numerical examples. We conclude in Section~\ref{section_conclusion}.

\section{WKB approximation}\label{section:wkb_approximation}
In this section we introduce the \textit{WKB approximation} as an approximate solution of IVP (\ref{schroedinger_eq}). The basis for its construction is the well-known WKB-ansatz (cf.\ \cite{Bender1999AdvancedMM,Landau1985Q}), which for the ODE (\ref{schroedinger_eq}) reads\footnote{We say that two functions $f,g:I\times(0,\varepsilon_{0})\to \mathbb{C}$ are \textit{asymptotically equivalent} as $\varepsilon\to 0$, if and only if for any $x\in I$ it holds $f(x,\varepsilon)-g(x,\varepsilon)=o(g(x,\varepsilon))$ as $\varepsilon\to 0$. In this case we write $f(x,\varepsilon)\sim g(x,\varepsilon)$, $\varepsilon\to 0$.}
\begin{align}\label{WKB-ansatz}
	\varphi(x)\sim \exp\left(\frac{1}{\varepsilon}S(x)\right)\Comma\quad \varepsilon \to 0\Comma	
\end{align}
where $S$ is a complex-valued function containing information of the phase as well as the amplitude of the solution $\varphi$. To derive WKB approximations it is then convenient to express $S$ as an asymptotic expansion\footnote{We say that a function $f:I\times(0,\varepsilon_{0})\to\mathbb{C}$ has an \textit{asymptotic expansion} as $\varepsilon\to 0$, if and only if there exist sequences of functions $\left(f_{n}:I\to\mathbb{C}\right)_{n\in\mathbb{N}_{0}}$ and ($\phi_{n}:(0,\varepsilon_{0})\to\mathbb{C})_{n\in\mathbb{N}_{0}}$ satisfying for all $n\in\mathbb{N}_{0}$ and $x\in I$ that $\phi_{n+1}(\varepsilon)f_{n+1}(x)=o(\phi_{n}(\varepsilon)f_{n}(x))$ as $\varepsilon\to 0$, such that for all $N\geq 0$ it holds $f(x,\varepsilon)-\sum_{n=0}^{N}\phi_{n}(\varepsilon)f_{n}(x)=o(\phi_{N}(\varepsilon)f_{N}(x))$ as $\varepsilon\to 0$. In this case we write $f(x,\varepsilon)\sim\sum_{n=0}^{\infty}\phi_{n}(\varepsilon)f_{n}(x)$, $\varepsilon\to 0$. We call an asymptotic expansion \textit{uniform} w.r.t.\ $x\in I$, if all the order symbols hold uniformly in $x\in I$.} w.r.t.\ the small parameter $\varepsilon$: 
\begin{align}\label{wkb_S}
	S(x)&\sim \sum_{n=0}^{\infty}\varepsilon^{n}S_{n}(x)\Comma\quad \varepsilon \to 0\,;\quad S_{n}(x)\in\mathbb{C}\period
\end{align}
It should be noted that this asymptotic series is typically divergent (as usual for asymptotic series) and must therefore be truncated in order to obtain an approximate solution.

By substituting the ansatz (\ref{WKB-ansatz})-(\ref{wkb_S}) into (\ref{schroedinger_eq}), one obtains (formally)
\begin{align}
	\left(\sum_{n=0}^{\infty}\varepsilon^{n}S_{n}^{\prime}(x)\right)^{2}+\sum_{n=0}^{\infty}\varepsilon^{n+1}S_{n}^{\prime\prime}(x)+a(x)=0.
\end{align}
A comparison of $\varepsilon$-powers then yields the following well-known recurrence relation for the functions $S_{n}'$:
\begin{align}
	S_{0}^{\prime}&=\pm \ii \sqrt{a}\Comma\label{S0}\\
	S_{1}^{\prime}&=-\frac{S_{0}^{\prime\prime}}{2S_{0}^{\prime}}=-\frac{a^{\prime}}{4a}=-\frac{1}{4}(\ln(a))'\Comma\label{S1}\\
	S_{n}^{\prime}&=-\frac{1}{2S_{0}^{\prime}}\left(\sum_{j=1}^{n-1}S_{j}^{\prime}S_{n-j}^{\prime}+S_{n-1}^{\prime\prime}\right)\Comma\quad n\geq 2\label{Sn}\period
\end{align}
The computation of each $S_{n}$, $n\geq 0$, thus involves one integration constant. Further, the repeated differentiation in (\ref{Sn}) indicates that a WKB approximation relying on $N+1$ terms in the truncated series (\ref{wkb_S}) requires $a\in C^{N}(I)$. Moreover, the fact that the r.h.s.\ of (\ref{S0}) has two different signs implies that there are two sequences of functions, which solve (\ref{S0})-(\ref{Sn}). This corresponds to the fact that there are two fundamental solutions of the ODE in (\ref{schroedinger_eq}). Let us denote by $(S_{n}^{-})_{n\in\mathbb{N}_{0}}$ the sequence induced by the choice $S_{0}^{\prime}=- \ii \sqrt{a}$. The one following from $S_{0}^{\prime}= \ii \sqrt{a}$ will be denoted by $(S_{n}^{+})_{n\in\mathbb{N}_{0}}$. Then, a simple observation is the following proposition.
\begin{Proposition}\label{proposition_S_and_S_tilde}
	\begin{align}
		(S_{2n}^{-})^{\prime}(x)&=-(S_{2n}^{+})^{\prime}(x)\in\ii\mathbb{R}\Comma \label{proposition_S_and_S_tilde_eq1}\\ (S_{2n+1}^{-})^{\prime}(x)&=(S_{2n+1}^{+})^{\prime}(x)\in\mathbb{R}\label{proposition_S_and_S_tilde_eq2}\Comma
	\end{align}
	for all $x\in I$ and $n\geq 0$.
\end{Proposition}
\begin{proof}
	The statement can easily be verified by induction on $n\in\mathbb{N}_{0}$.
\end{proof}
Since both sequences $(S_{n}^{\pm})_{n\in\mathbb{N}_{0}}$ lead to an approximate solution of the ODE in (\ref{schroedinger_eq}), the general approximate solution can be written as the linear combination
\begin{align}\label{general_wkb_solution}
\varphi\approx\varphi_{N}^{WKB}:=\alpha_{N,\varepsilon}\exp\left(\sum_{n=0}^{N}\varepsilon^{n-1}S_{n}^{-}\right)+\beta_{N,\varepsilon}\exp\left(\sum_{n=0}^{N}\varepsilon^{n-1}S_{n}^{+}\right)\Comma
\end{align}
with arbitrary $\alpha_{N,\varepsilon},\beta_{N,\varepsilon}\in\mathbb{C}$. Note that all integration constants in the computation of $S_{n}^{-}$ and $S_{n}^{+}$ can be ``absorbed'' into $\alpha_{N,\varepsilon}$ and $\beta_{N,\varepsilon}$, respectively. Hence, these integration constants can be set
to zero without loss of generality.
More precisely, we define
\begin{align}\label{Sn_definition}
	S_{n}^{\pm}(x):=\int_{\xi}^{x}\left(S_{n}^{\pm}\right)'(\tau)\,\mathrm{d}\tau\period
\end{align}
With this, Proposition~\ref{proposition_S_and_S_tilde} implies
\begin{align}
		S_{2n}^{-}(x)&=-S_{2n}^{+}(x)\in\ii\mathbb{R}\Comma \label{proposition_S_and_S_tilde_implication1}\\ S_{2n+1}^{-}(x)&=S_{2n+1}^{+}(x)\in\mathbb{R}\Comma\label{proposition_S_and_S_tilde_implication2}
\end{align}
for all $x\in I$ and $n\geq 0$. Hence, functions with even indices only contribute to the phase of the WKB approximation $\varphi_{N}^{WKB}$, whereas functions with odd indices only provide corrections to the amplitude.

Note that in general the constants $\alpha_{N,\varepsilon}$ and $\beta_{N,\varepsilon}$ can be uniquely determined by initial or boundary conditions. Here, for the WKB approximation (\ref{general_wkb_solution}) to satisfy the initial conditions in (\ref{schroedinger_eq}),
%
we set
\begin{align}
	\alpha_{N,\varepsilon}&=\frac{\varphi_{0}\left(\sum_{n=0}^{N}\varepsilon^{n}(S_{n}^{+})^{\prime}(\xi)\right)-\varphi_{1}}{\sum_{n=0}^{N}\varepsilon^{n}\left((S_{n}^{+})^{\prime}(\xi)-(S_{n}^{-})^{\prime}(\xi)\right)}\Comma\label{alpha_N_eps}\\
	\beta_{N,\varepsilon}&=\frac{\varphi_{1}-\varphi_{0}\left(\sum_{n=0}^{N}\varepsilon^{n}(S_{n}^{-})^{\prime}(\xi)\right)}{\sum_{n=0}^{N}\varepsilon^{n}\left((S_{n}^{+})^{\prime}(\xi)-(S_{n}^{-})^{\prime}(\xi)\right)}\period\label{beta_N_eps}
\end{align}
In the following we will often simply write $S_n$ whenever one could insert either $S_{n}^{-}$ or $S_{n}^{+}$.

According to \cite[Sec.\ 10.2]{Bender1999AdvancedMM}, for the WKB-ansatz (\ref{WKB-ansatz})-(\ref{wkb_S}) to be valid on the whole interval $I$, it is necessary that the series $\sum_{n=0}^{\infty}\varepsilon^{n-1}S_{n}(x)$ is a uniform asymptotic expansion of $\varepsilon^{-1}S(x)$ as $\varepsilon\to 0$. 
%
%
This implies that for any $n\in\mathbb{N}_{0}$ the relation
\begin{align}
    \varepsilon^{n}S_{n+1}(x)=o(\varepsilon^{n-1}S_{n}(x))\Comma\quad \varepsilon\to 0\Comma
\end{align}
must hold uniformly in $x\in I$.
Note that this condition is violated if the interval $I$ includes so-called \textit{turning points}, i.e., points $x_{0}\in I$ with $a(x_{0})=0$. Indeed, this is already evident from (\ref{S1}), which implies that $S_{1}$ blows up at such turning points.

\section{Error analysis}\label{section_error_analysis}
In this section we aim to find an explicit (w.r.t.\ $\varepsilon$ and the truncation order $N$) error estimate for the WKB approximation (\ref{general_wkb_solution}). One key ingredient will be the following a priori estimate for the solution $\varphi$ of the inhomogeneous analog of the Schrödinger equation-IVP (\ref{schroedinger_eq}).
\begin{Proposition}\label{proposition_apriori}
	 Let $a\in W^{1,\infty}(I)$ with $a(x)\geq a_{0}>0$ and $f\in C(I)$. Further, let $\varphi\in C^{2}(I)$ be the solution of the inhomogeneous IVP
	\begin{align}
		\begin{cases}
			\varepsilon^{2}\varphi''+a(x)\varphi=f(x)\Comma \quad x\in I\Comma\nonumber\\
			\varphi(\xi)=\hat{\varphi}_{0}\Comma\nonumber\\
			\varepsilon\varphi'(\xi)=\hat{\varphi}_{1}\Comma
		\end{cases}
	\end{align}
 with constants $\hat{\varphi}_{0},\hat{\varphi}_{1}\in \mathbb{C}$. Then there exists $C>0$ independent of $\varepsilon$ such that \anton{it holds for all $f\in C(I)$ and $\hat{\varphi}_{0},\hat{\varphi}_{1}\in \mathbb{C}$: }
	\begin{align}
		\lVert \varphi\rVert_{L^{\infty}(I)}&\leq \frac{C}{\varepsilon}\lVert f\rVert_{L^{2}(I)}+C\left(|\hat{\varphi}_{1}|+|\hat{\varphi}_{0}|\right)\Comma \label{est_apriori_1}\\
		\lVert \varepsilon\varphi'\rVert_{L^{\infty}(I)}&\leq \frac{C}{\varepsilon}\lVert f\rVert_{L^{2}(I)}+C\left(|\hat{\varphi}_{1}|+|\hat{\varphi}_{0}|\right)\period\label{est_apriori_2}
	\end{align}  
\end{Proposition}
\begin{proof}
	Estimates (\ref{est_apriori_1})-(\ref{est_apriori_2}) can be derived by finding an upper bound for the real-valued function $E(x):=\varepsilon^{2}|\varphi'|^{2}+a|\varphi|^{2}$. At first, it holds that
	\begin{align}\label{est_E_1}
		\frac{\mathrm{d}}{\mathrm{d}x}E(x)&=\varepsilon^{2}\frac{\mathrm{d}}{\mathrm{d}x}|\varphi'|^{2}+a\frac{\mathrm{d}}{\mathrm{d}x}|\varphi|^{2}+a'|\varphi|^2\nonumber\\
		&=2\operatorname{Re}((\varepsilon^{2}\varphi''+a\varphi)\overline{\varphi'})+a'|\varphi|^2\nonumber\\
		&=2\operatorname{Re}(f\overline{\varphi'})+a'|\varphi|^{2}\nonumber\\
		&\leq 2|f||\varphi'|+\lVert a'\rVert_{L^{\infty}(I)}|\varphi|^{2}\period
	\end{align}
	Using Young's inequality, we obtain
	\begin{align}\label{est_young}
		2|f||\varphi'|\leq \frac{1}{\varepsilon^{2}}|f|^2+\varepsilon^{2}|\varphi'|^{2}\period
	\end{align}
	Moreover, $a(x)\geq a_{0} > 0$ implies that
	\begin{align}\label{est_a_0}
		\lVert a'\rVert_{L^{\infty}(I)}|\varphi|^{2}\leq \frac{\lVert a'\rVert_{L^{\infty}(I)}}{a_{0}}a|\varphi|^{2}\period
	\end{align}
	Thus, from (\ref{est_E_1})-(\ref{est_a_0}) we obtain with $c:=\max(1,\frac{\lVert a'\rVert_{L^{\infty}(I)}}{a_{0}})\geq 1$
	\begin{align}
 \label{eq:gronwall-1}
		\frac{\mathrm{d}}{\mathrm{d}x}E(x)&\leq \frac{1}{\varepsilon^{2}} \anton{|f(x)|^{2}} +cE(x)\period
	\end{align}
	%
    Applying Gronwall's inequality \markus{(i.e., multiply \eqref{eq:gronwall-1} by $\exp(-cx)$ and integrate)}, we therefore get
	\begin{align}
		E(x)& \anton{ \leq \frac{1}{\varepsilon^{2}}\int_\xi^x |f(s)|^{2} \ee^{c(x-s)} \mathrm{d}s +E(\xi)\ee^{c(x-\xi)} } \nonumber\\
		&\leq \left(\frac{1}{\varepsilon^{2}}\lVert f\rVert_{L^{2}(I)}^{2}+E(\xi)\right)\ee^{c(x-\xi)}\nonumber\\
		&\leq \ee^{c(\eta-\xi)}\left(\frac{1}{\varepsilon^{2}}\lVert f\rVert_{L^{2}(I)}^{2}+|\hat{\varphi}_{1}|^{2}+a(\xi)|\hat{\varphi}_{0}|^{2}\right)\Comma
	\end{align}
	which implies the estimates (\ref{est_apriori_1})-(\ref{est_apriori_2}).
\end{proof}
%
In order to derive an error estimate for the WKB approximation (\ref{general_wkb_solution}), which is explicit not only w.r.t.\ $\varepsilon$ but also w.r.t.\ the truncation order $N$, it is essential to control the growth of the functions $S_{n}$ w.r.t.\ $n\in\mathbb{N}_{0}$. As a first step, we aim to establish upper bounds for the derivatives $S_{n}'$ which are given by recurrence relation (\ref{S0})-(\ref{Sn}). To this end, we employ a strategy similar to \cite[Lemma 2]{Melenk1997OnTR}, which relies heavily on Cauchy's integral formula. To enable us to apply this tool, we shall assume that $S_{0}'$ is not only defined on the real interval $I$, but also on a complex neighbourhood $G\subset \mathbb{C}$ of $I$ \markus{ and holomorphic there.} This leads us to introduce the following assumption.
\begin{Hypothesis}{A}\label{hypothesis_A}
    Let $S_{0}'$ be \markus{holomorphic (complex analytic)} on a complex, bounded, simply connected neighbourhood $G\subset\mathbb{C}$ of $I$, satisfying $S_{0}'(z)\neq 0$ for any $z\in G$.
\end{Hypothesis}
As a consequence of Hypothesis~\ref{hypothesis_A}, the function $a$ and all $S_{n}$, $n\in\mathbb{N}$, are \markus{holomorphic}  on $G$. In particular, each $S_{n}$ is bounded on $I$.

For the next lemma, we introduce, \markus{for $\delta > 0$}, the open sets
\begin{align}\label{definition_G_delta}
	G_{\delta}:=\{z\in G \mid \operatorname{dist}(z,\partial G)> \delta\}.
\end{align}
%
%
\begin{Lemma}\label{lemma_est_Sn'}
	Let Hypothesis~\ref{hypothesis_A} be satisfied and let $0<\delta\leq 1$ be such that $G_{\delta}\neq \emptyset$. Then there exists 
    \anton{a constant $K>0$ }
    depending only on $G$ and $S_{0}'$ such that
	\begin{align}
		\lVert S_{n}'\rVert_{L^{\infty}(G_{\delta})}\leq \lVert S_{0}'\rVert_{L^{\infty}(G)} \anton{K^n} 
        n^{n}\delta^{-n}\Comma\quad n\in\mathbb{N}_{0}\period\label{K_delta_est}
	\end{align}
    Here, we define $0^{0}$ as $1$.
\end{Lemma}
\begin{proof}
	Define the auxiliary functions $\widehat{S}_{n}':=-(2S_{0}')^{-1}S_{n}'$. By using (\ref{S0})-(\ref{Sn}) we then find that the functions $\widehat{S}_{n}'$ satisfy the following recurrence relation
	\begin{align}
		\widehat{S}_{0}'&=-\frac{1}{2}\Comma\label{hatS0}\\
		\widehat{S}_{n}'&=\left(\sum_{j=1}^{n-1}\widehat{S}_{j}'\widehat{S}_{n-j}'\right)+\left(2S_{0}'\right)^{-2}(-2S_{0}'\widehat{S}_{n-1}')'\Comma\quad n\geq 1 \label{hatSn}\period
	\end{align}
	Note that since $S_{0}'$ is \markus{holomorphic}  on $G$, it follows from recurrence relation (\ref{S0})-(\ref{Sn}) that $S_{n}'$, and hence also $\widehat{S}_{n}'$, is \markus{holomorphic}  on $G$, for every $n\in\mathbb{N}_{0}$.
	We will now prove by induction on $n$ that
	\begin{align}\label{est_hat_S_n'}
		\lVert \widehat{S}_{n}'\rVert_{L^{\infty}(G_{\delta})}\leq \frac{1}{2} \anton{K^n}
        n^{n}\delta^{-n}\Comma\quad n\in\mathbb{N}_{0}\period
	\end{align}
	Obviously, this estimate does hold for $n=0$, according to (\ref{hatS0}). Assume now that the estimate in (\ref{est_hat_S_n'}) holds for $0\leq j\leq n-1$ with some fixed $n\geq 1$. We will now prove it for $n$.
	Let $0<\kappa < 1$ and $z\in G_{\delta}$. We denote with $\partial B_{\kappa\delta}(z)$ a circle of radius $\kappa\delta$ around $z$, see the left \markus{panel} of Figure~\ref{plot:G_and_G_delta}.
	\begin{figure}
		\centering
		\subfloat[]{\includegraphics[trim=35 30 35 40,clip,scale=1.937]{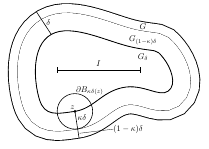}}
        \subfloat[]{\includegraphics[trim=35 30 35 40,clip,scale=1.937]{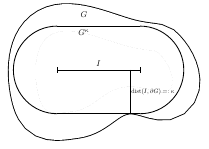}}
        \caption{(a) Exemplary sketch of the situation from the proof of Lemma~\ref{lemma_est_Sn'}: $G_{\delta}\subset G_{(1-\kappa)\delta}\subset G$, where $G$ is a complex neighbourhood of the interval $I$. Here, the point $z\in G_{\delta}$ is very close to the boundary $\partial G_{\delta}$, which makes it clear why one has to consider $G_{(1-\kappa)\delta}$ in the r.h.s.\ of (\ref{hatS_n-1}). (b) Every possible candidate $G$ for the minimum on the l.h.s.\ of (\ref{G_minimization}) can be reduced to a set $G^{\kappa}:=\{z\in\mathbb{C} \mid \operatorname{dist}(z,I)<\kappa\}$, where $\kappa:=\operatorname{dist}(I,\partial G)>0$.}
		\label{plot:G_and_G_delta}
	\end{figure}
	Then Cauchy's integral formula implies
	\begin{align}
		|(-2S_{0}'\widehat{S}_{n-1}')'(z)|&=\frac{1}{2\pi}\left|\int_{\partial B_{\kappa\delta}(z)}\frac{-2S_{0}'(\zeta)\widehat{S}_{n-1}'(\zeta)}{(\zeta-z)^{2}}\,\mathrm{d}\zeta\right|\nonumber\\
		&\leq \frac{2\pi\kappa\delta}{2\pi}2\lVert S_{0}'\rVert_{L^{\infty}(G)}\lVert \widehat{S}_{n-1}'\rVert_{L^{\infty}(\partial B_{\kappa\delta}(z))}(\kappa\delta)^{-2}\period
	\end{align}
	This, together with the fact that $\partial B_{\kappa\delta}(z)\subseteq\markus{\overline{G_{(1-\kappa)\delta}}}$ yields
	\begin{align}\label{hatS_n-1}
		\lVert (-2S_{0}'\widehat{S}_{n-1}')'\rVert_{L^{\infty}(G_{\delta})}\leq 2\lVert S_{0}'\rVert_{L^{\infty}(G)}(\kappa\delta)^{-1}\lVert \widehat{S}_{n-1}'\rVert_{L^{\infty}(G_{(1-\kappa)\delta})}\period
	\end{align}
	By applying estimate (\ref{hatS_n-1}) and the induction hypothesis to (\ref{hatSn}), we find
	\begin{align}\label{est_hatSn_first}
		\lVert \widehat{S}_{n}'\rVert_{L^{\infty}(G_{\delta})}&\leq \sum_{j=1}^{n-1}\lVert \widehat{S}_{j}'\rVert_{L^{\infty}(G_{\delta})}\lVert \widehat{S}_{n-j}'\rVert_{L^{\infty}(G_{\delta})}\nonumber\\
        &\quad+\frac{1}{4}\lVert (S_{0}')^{-2}\rVert_{L^{\infty}(G_{\delta})}2\lVert S_{0}'\rVert_{L^{\infty}(G)}(\kappa\delta)^{-1}\lVert \widehat{S}_{n-1}'\rVert_{L^{\infty}(G_{(1-\kappa)\delta})}\nonumber\\
		&\leq \frac{1}{4} \anton{K^n}
        \delta^{-n}\sum_{j=1}^{n-1}j^{j}(n-j)^{n-j}\nonumber\\
    &\quad+\frac{1}{4}\lVert (S_{0}')^{-2}\rVert_{L^{\infty}(G_{\delta})}\lVert S_{0}'\rVert_{L^{\infty}(G)}\delta^{-n}
    \anton{K^{n-1}} \frac{(n-1)^{n-1}}{\kappa(1-\kappa)^{n-1}}\period
	\end{align}
	Since $j^{j}(n-j)^{n-j}\leq (n-1)^{n-1}$ for all $1\leq j \leq n-1$, we can bound the sum in the first term of (\ref{est_hatSn_first}) by $(n-1)^{n}$. Thus, we obtain
	\begin{align}
		\lVert \widehat{S}_{n}'\rVert_{L^{\infty}(G_{\delta})}&\leq \frac{1}{2} \anton{K^n}
        n^{n}\delta^{-n}\left[\frac{1}{2}\left(\frac{n-1}{n}\right)^{n}+\frac{\lVert (S_{0}')^{-2}\rVert_{L^{\infty}(G_{\delta})}\lVert S_{0}'\rVert_{L^{\infty}(G)}}{2\anton{K}
        \kappa(1-\kappa)^{n-1}n}\left(\frac{n-1}{n}\right)^{n-1}\right]\period
	\end{align}
	It now suffices to show that the expression in the square brackets is less than or equal to $1$. By further estimating $\left(\frac{n-1}{n}\right)^{n}\leq \frac{1}{\ee}$, and choosing $\kappa=\frac{1}{n}$, we get
	\begin{align}
		\lVert \widehat{S}_{n}'\rVert_{L^{\infty}(G_{\delta})}&\leq\frac{1}{2}\anton{K^n}
        n^{n}\delta^{-n}\left[\frac{1}{2\ee}+\frac{\lVert (S_{0}')^{-2}\rVert_{L^{\infty}(G_{\delta})}\lVert S_{0}'\rVert_{L^{\infty}(G)}}{2\anton{K}
        }\right]
	\end{align}
	Thus it is sufficient to choose
 \begin{align}\label{K_definition}
     \anton{K:=\frac{\ee}{2\ee-1}\lVert (S_{0}')^{-2}\rVert_{L^{\infty}(G)} } \lVert S_{0}'\rVert_{L^{\infty}(G)}\period
 \end{align}
 %
\anton{The estimate $\lVert (S_{0}')^{-2}\rVert_{L^{\infty}(G_{\delta})}\le \lVert (S_{0}')^{-2}\rVert_{L^{\infty}(G)}$} 
concludes the proof.
\end{proof}
A simple but important implication of Lemma~\ref{lemma_est_Sn'} is the fact that we are now able to provide estimates not only for all the derivatives of $S_{n}$ but also for $S_{n}$ itself:
\begin{Corollary}\label{corollary_Snk}
	Let Hypothesis~\ref{hypothesis_A} be satisfied. 
	Then there exist constants $K_{1},K_{2}>0$ depending only on $G$ 
	and $S_{0}'$ such that 
	\begin{align}
		\lVert S_{n}\rVert_{L^{\infty}(I)}&\leq (\eta-\xi)\lVert S_{0}'\rVert_{L^{\infty}(G)}K_{2}^{n}n^{n}\Comma\quad n\in\mathbb{N}_{0}\Comma\label{Sn_est}\\
        \lVert S_{n}^{(k)}\rVert_{L^{\infty}(I)}&\leq
		\lVert S_{0}'\rVert_{L^{\infty}(G)}(k-1)!\,K_{1}^{k-1}K_{2}^{n}n^{n}\Comma\quad n\in\mathbb{N}_{0}\Comma\quad k\in\mathbb{N}\period\label{Snk_est}
	\end{align}
    Here we define $0^{0}$ as $1$.
\end{Corollary}
\begin{proof}
	Since $G$ is a complex neighbourhood of $I$, there is some $0<\delta\leq 1$ such that $I\subset G_{2\delta}$. To prove estimate (\ref{Snk_est}), we start with the trivial estimate $\lVert S_{n}^{(k)}\rVert_{L^{\infty}(I)}\leq\lVert S_{n}^{(k)}\rVert_{L^{\infty}(G_{2\delta})}$. Then, for any $k\in\mathbb{N}$ and $z\in G_{2\delta}$, Cauchy's integral formula implies 
 \begin{align}\label{corollary_cauchy}
        |S_{n}^{(k)}(z)|=\frac{(k-1)!}{2\pi}\left|\int_{\partial B_{\delta}(z)}\frac{S_{n}'(z)}{(\zeta-z)^{k}}\,\mathrm{d}\zeta\right|\leq (k-1)!\,\delta^{-k+1}\lVert S_{n}'\rVert_{L^{\infty}(G_{\delta})}\period
 \end{align}
 By applying Lemma~\ref{lemma_est_Sn'} on the r.h.s.\ of (\ref{corollary_cauchy}) we conclude that (\ref{Snk_est}) holds with $K_{1}:=1/\delta$ and $K_{2}:= \anton{K/\delta}$.
 Estimate (\ref{Sn_est}) then follows from (\ref{Snk_est}) for $k=1$ and by the definition of $S_n$, see (\ref{Sn_definition}).
\end{proof}
\begin{Remark}\label{remark:K_2_trade_off}
    Of course, it is of great interest to find a constant $K_{2}$ from Corollary~\ref{corollary_Snk}  is as small as possible. To this end one would have to minimize the constant \anton{$K/\delta$}
    in estimate (\ref{K_delta_est}). In particular, one has to fix some complex neighbourhood $G$ of $I$ as well as a constant $0<\delta\leq 1$ such that it holds $I\subset G_{\delta}$. Further, the proof of Lemma~\ref{lemma_est_Sn'} indicates that \anton{$K$}
    can be reduced by choosing $G$ small, see (\ref{K_definition}). However, this means that one is forced to reduce also the value of $\delta$. Hence, this procedure usually results in a trade-off between the magnitudes of \anton{$K$}
    and $\delta$. More precisely, one \markus{is lead} to solve the following minimization problem:
    \begin{align}\label{G_minimization}
        \anton{
        \inf_{\substack{0<\delta\leq 1 \\ G\subset \mathbb{C} \\ I\subset G_{\delta}}}\frac{\lVert (S_{0}')^{-2}\rVert_{L^{\infty}(G)}\lVert S_{0}'\rVert_{L^{\infty}(G)}}{\delta}
        =\min_{0<\delta\leq 1 }
        \frac{\lVert (S_{0}')^{-2}\rVert_{L^{\infty}(G^\delta)} \lVert S_{0}'\rVert_{L^{\infty}(G^{\delta})}}{\delta}\Comma  }
        \end{align}
    where $G^{\delta}:=\{z\in\mathbb{C} \mid \operatorname{dist}(z,I)<\delta\}$. Equality in (\ref{G_minimization}) holds for the following reasons: First, on the l.h.s.\ of (\ref{G_minimization}) one only needs to consider sets $G\subset\mathbb{C}$ of the form $G=G^{\kappa}:=\{z\in\mathbb{C} \mid \operatorname{dist}(z,I)<\kappa\}$, with $\kappa>0$. For any $0<\delta\leq 1$ such that $I\subset G_{\delta}$, this follows since the numerator on the l.h.s.\ of (\ref{G_minimization}) is not increased when 
    \anton{
    replacing $G$ 
    by $G^{\operatorname{dist}(I,\partial G)}$,  
    %
    %
    see the right panel of Figure~\ref{plot:G_and_G_delta}.
    The condition $I\subset G_{\delta}$ then simply reads $\kappa>\delta$.
    %
    %
    Second, since for a fixed $0<\delta\leq 1$ and $\delta<\kappa_{1}<\kappa_{2}$ it holds that $G^{\kappa_{1}}\subset G^{\kappa_{2}}$, 
    it is sufficient to consider simply the sets $G^{\delta+\epsilon}$, with $\epsilon>0$ being an arbitrarily small number. The equality in (\ref{G_minimization}) then follows from the fact that $\markus{\bigcap_{\epsilon> 0}G^{\delta+\epsilon}=\overline{G^{\delta}}}$. 
    }
\end{Remark}
We will later make use of the residual of the WKB approximation (\ref{general_wkb_solution}) w.r.t.\ the ODE in (\ref{schroedinger_eq}). For this, the following lemma will be helpful.
\begin{Lemma}\label{lemma_residual}
	Denote with $L_{\varepsilon}:=\varepsilon^{2}\frac{\dd^{2}}{\dd x^{2}}+a(x)$ the linear operator appearing in the Schrödinger equation (\ref{schroedinger_eq}) and let $\widetilde{\varphi}_{N}:=\exp\left(\sum_{n=0}^{N}\varepsilon^{n-1}S_{n}\right)$, $N\in\mathbb{N}_{0}$. Then it holds
	\begin{align}\label{residual_equation}
	   L_{\varepsilon}\widetilde{\varphi}_{N}=\widetilde{\varphi}_{N}f_{N,\varepsilon}\Comma
	\end{align}
	where
	\begin{align}\label{residual_f_eps}
		f_{N,\varepsilon}=\varepsilon^{N+1}(-2S_{0}'S_{N+1}')+\sum_{n=2}^{N}\sum_{k=2+N-n}^{N}\varepsilon^{n+k}S_{n}'S_{k}'\semicolon
	\end{align}
    for $N<2$ the double sum is dropped.
\end{Lemma}
\begin{proof}
	First we observe that
	\begin{align}\label{residual_step_1}
		L_{\varepsilon}\widetilde{\varphi}_{N}&=\varepsilon^{2}\widetilde{\varphi}_{N}^{\prime\prime}+a(x)\widetilde{\varphi}_{N}\nonumber\\
		&=\varepsilon^{2}\widetilde{\varphi}_{N}\left(\left(\frac{1}{\varepsilon}\sum_{n=0}^{N}\varepsilon^{n}S_{n}'\right)^{2}+\frac{1}{\varepsilon}\sum_{n=0}^{N}\varepsilon^{n}S_{n}''\right)+a(x)\widetilde{\varphi}_{N}\nonumber\\
		&=\widetilde{\varphi}_{N}\left( \anton{\sum_{n=0}^N\sum_{k=0}^N\varepsilon^{n+k}S_{n}'S_{k}' }+\sum_{n=0}^{N}\varepsilon^{n+1}S_{n}''+a(x)\right)\period
	\end{align}
	Let us denote the second factor in (\ref{residual_step_1}) by $f_{N,\varepsilon}$. We will now show that $f_{N,\varepsilon}$ reduces to (\ref{residual_f_eps}). To this end, let us first rewrite $f_{N,\varepsilon}$ as
	\begin{align}\label{residual_step_2}
		f_{N,\varepsilon}&=\left(S_{0}'^{2}+a\right)+\Bigg( \anton{\sum_{n=0}^N \sum_{k=\max(0,1-n)}^{N-n}\varepsilon^{n+k}S_{n}'S_{k}' }+
    \sum_{n=0}^{N-1}\varepsilon^{n+1}S_{n}''\Bigg)\nonumber\\
		&\quad+\Bigg(\anton{ \sum_{n=1}^{N} \varepsilon^{N+1}S_{N+1-n}'S_{n}' }
    +\varepsilon^{N+1}S_{N}''\Bigg)+ \anton{ \sum_{n=2}^N \sum_{k=N+2-n}^N \varepsilon^{n+k}S_{n}'S_{k}'}
    \period
	\end{align}
	Now, the first term in (\ref{residual_step_2}) vanishes due to (\ref{S0}). The second term also vanishes since
	\begin{align}
        \anton{\sum_{n=0}^N \sum_{k=\max(0,1-n)}^{N-n}\varepsilon^{n+k}S_{n}'S_{k}' }
        &=\sum_{n=0}^{N-1}\varepsilon^{n+1}\sum_{j=0}^{n+1}S_{j}'S_{n+1-j}'\nonumber\\
		&=\sum_{n=0}^{N-1}\varepsilon^{n+1}\left(2S_{0}'S_{n+1}'+\sum_{j=1}^{n}S_{j}'S_{n+1-j}'\right)\nonumber\\
		&=-\sum_{n=0}^{N-1}\varepsilon^{n+1}S_{n}''\Comma
	\end{align}
	where we used in the last equation recurrence relation (\ref{Sn}) for the function $S_{n+1}'$. Finally, by using (\ref{Sn}) for the function $S_{N+1}'$, the third term in (\ref{residual_step_2}) simplifies to $\varepsilon^{N+1}(-2S_{0}'S_{N+1}')$. The claim follows.
\end{proof}
Recalling that $S_{0}(x)\in \ii\mathbb{R}$ we note that $\widetilde{\varphi}_{N}(x)$ is, for fixed $x\in I$, uniformly bounded w.r.t.\ $\varepsilon\in (0,1]$. Thus the r.h.s.\ of (\ref{residual_equation}) is of the order $\mathcal{O}(\varepsilon^{N+1})$, and we conclude from Lemma~\ref{lemma_residual} that the function $\widetilde{\varphi}_{N}$ satisfies the ODE $L_{\varepsilon}\varphi=0$ asymptotically, as $\varepsilon\to 0$. This is one of the main properties we can utilize to show that also the numerical error of the WKB approximation (\ref{general_wkb_solution}) will approach $0$ as $\varepsilon\to 0$, at least for $N\geq1$. To this end we need the following lemma.
\begin{Lemma}\label{lemma_alpha_N_eps_phi_est}
	Let Hypothesis~\ref{hypothesis_A} be satisfied and define $\varphi_{N}^{\pm}:=\exp\left(\sum_{n=0}^{N}\varepsilon^{n-1}S_{n}^{\pm}\right)$, $N\in\mathbb{N}_{0}$. Then there exist constants $\varepsilon_{0}\in(0,1)$ and $C>0$ such that it holds for $\varepsilon\in(0,\varepsilon_{0}]$:
	\begin{align}
		\lVert \alpha_{N,\varepsilon}\varphi_{N}^{-}\rVert_{L^{\infty}(I)}&\leq C\left(|\varphi_{0}|\,\lVert S_{0}'\rVert_{L^{\infty}(G)}\sum_{n=0}^{N}\varepsilon^{n}K_{2}^{n}n^{n}+|\varphi_{1}|\right)\nonumber\\
		&\quad\times\exp\left((\eta-\xi)\lVert S_{0}'\rVert_{L^{\infty}(G)}\sum_{n=0}^{\lfloor \frac{N-1}{2}\rfloor}\varepsilon^{2n}K_{2}^{2n+1}(2n+1)^{2n+1}\right)\Comma
	\end{align}
    with $\alpha_{N,\varepsilon}$ from (\ref{alpha_N_eps}). For $N=0$ the last sum is dropped. The same estimate holds for $\lVert \beta_{N,\varepsilon}\varphi_{N}^{+}\rVert_{L^{\infty}(I)}$.
	In particular, since the initial values $\varphi_{0}$ and $\varphi_{1}$ are assumed to be uniformly bounded w.r.t.\ $\varepsilon$, so is $\varphi_{N}^{WKB}$ in $L^{\infty}(I)$.
\end{Lemma}
\begin{proof}
	We will prove only the estimate for $\alpha_{N,\varepsilon}\varphi_{N}^{-}$. For $\beta_{N,\varepsilon}\varphi_{N}^{+}$ it is fully analogous. First notice that Proposition~\ref{proposition_S_and_S_tilde} implies that
	\begin{align}
		|\alpha_{N,\varepsilon}|=\frac{\left|\varphi_{0}\sum_{n=0}^{N}\varepsilon^{n}(S_{n}^{+})'(\xi)-\varphi_{1}\right|}{2\left|\sum_{n=0}^{\lfloor \frac{N}{2}\rfloor}\varepsilon^{2n}(S_{2n}^{+})'(\xi)\right|}\period
	\end{align}
	Due to $a(x)\geq a_{0} >0$, we have $|(S_{0}^{+})'(\xi)|\geq \sqrt{a_{0}}>0$. Thus, there exists $\varepsilon_{0}\in(0,1)$ sufficiently small such that
	\begin{align}
		\left|\sum_{n=0}^{\lfloor \frac{N}{2}\rfloor}\varepsilon^{2n}(S_{2n}^{+})'(\xi)\right|&\geq \left|(S_{0}^{+})'(\xi)\right|-\sum_{n=1}^{\lfloor \frac{N}{2}\rfloor}\varepsilon^{2n}\left|(S_{2n}^{+})'(\xi)\right|\nonumber\\
		&\geq \sqrt{a_{0}}-\sum_{n=1}^{\lfloor \frac{N}{2}\rfloor}\varepsilon^{2n}\left|(S_{2n}^{+})'(\xi)\right|\nonumber\\
        &\geq \frac{1}{2C}
	\end{align}
	for all $\varepsilon\in(0,\varepsilon_{0}]$ and some $C>0$ (since $(S_{2n}^{+})'$ is bounded on $I$). Hence, we obtain
	\begin{align}\label{alpha_N_eps_est1}
		|\alpha_{N,\varepsilon}|\leq C \left|\varphi_{0}\sum_{n=0}^{N}\varepsilon^{n}(S_{n}^{+})'(\xi)-\varphi_{1}\right|\period
	\end{align}
	Next, (\ref{proposition_S_and_S_tilde_implication1})-(\ref{proposition_S_and_S_tilde_implication2}) imply $|\varphi_{N}^{-}(x)|\leq\exp\left(\sum_{n=0}^{\lfloor \frac{N-1}{2}\rfloor}\varepsilon^{2n}\left|S_{2n+1}^{-}(x)\right|\right)$ for all $x\in I$. Together with (\ref{alpha_N_eps_est1}) this yields
	\begin{align}
	|\alpha_{N,\varepsilon}\varphi_{N}^{-}(x)| \leq C \left(|\varphi_{0}|\sum_{n=0}^{N}\varepsilon^{n}\left|(S_{n}^{+})'(\xi)\right|+|\varphi_{1}|\right)\exp\left(\sum_{n=0}^{\lfloor \frac{N-1}{2}\rfloor}\varepsilon^{2n}\left|S_{2n+1}^{-}(x)\right|\right)
	\end{align}
    for all $x\in I$. Applying Corollary~\ref{corollary_Snk} then yields the claim.
\end{proof}
Finally, we provide an error estimate for the WKB approximation (\ref{general_wkb_solution}).
\begin{Theorem}\label{theorem_wkb_error}
	Let Hypothesis~\ref{hypothesis_A} be satisfied and let $\varphi\in C^{2}(I)$ be the solution of IVP (\ref{schroedinger_eq}). There exist constants $\varepsilon_{0}\in(0,1)$ and $C>0$ independent of $N$ and $\varepsilon$ such that it holds for $\varepsilon\in(0,\varepsilon_{0}]$:
	\begin{align}\label{wkb_error}
		\lVert\varphi-\varphi_{N}^{WKB}\rVert_{L^{\infty}(I)}&\leq C\lVert S_{0}'\rVert_{L^{\infty}(G)}^{2}\left(|\varphi_{0}|\,\lVert S_{0}'\rVert_{L^{\infty}(G)}\sum_{n=0}^{N}\varepsilon^{n}K_{2}^{n}n^{n}+|\varphi_{1}|\right)\nonumber\\
		&\quad\times\exp\left((\eta-\xi)\lVert S_{0}'\rVert_{L^{\infty}(G)}\sum_{n=0}^{\lfloor \frac{N-1}{2}\rfloor}\varepsilon^{2n}K_{2}^{2n+1}(2n+1)^{2n+1}\right)\nonumber\\
		&\quad\times\Bigg(\varepsilon^{N}K_{2}^{N+1}(N+1)^{N+1}+\sum_{n=2}^{N}\sum_{k=2+N-n}^{N}\varepsilon^{n+k-1}K_{2}^{n+k}n^{n}k^{k}\Bigg)\period
	\end{align}
    For $N=0$ the sum in the exponential function is dropped, and for $N<2$ the double sum is dropped.
\end{Theorem}
\begin{proof}
	To compute the residual of the WKB approximation (\ref{general_wkb_solution}), we notice that $\varphi_{N}^{WKB}=\alpha_{N,\varepsilon}\varphi_{N}^{-}+\beta_{N,\varepsilon}\varphi_{N}^{+}$, where $\varphi_{N}^{\pm}=\exp\left(\sum_{n=0}^{N}\varepsilon^{n-1}S_{n}^{\pm}\right)$. By applying Lemma~\ref{lemma_residual}, we obtain
	\begin{align}
		L_{\varepsilon}(\varphi-\varphi_{N}^{WKB})&=-\alpha_{N,\varepsilon}L_{\varepsilon}\varphi_{N}^{-}-\beta_{N,\varepsilon}L_{\varepsilon}\varphi_{N}^{+}\nonumber\\
		&=-\alpha_{N,\varepsilon}\varphi_{N}^{-}f_{N,\varepsilon}^{-}-\beta_{N,\varepsilon}\varphi_{N}^{+}f_{N,\varepsilon}^{+}\Comma
	\end{align}
	where the functions $f_{N,\varepsilon}^{\pm}$ are given by (\ref{residual_f_eps}) when inserting $S_{n}^{\pm}$ for $S_{n}$. Further, since $\varphi_{N}^{WKB}$ satisfies the initial conditions in (\ref{schroedinger_eq}), we have $(\varphi-\varphi_{N}^{WKB})(\xi)=0$ and $\varepsilon(\varphi-\varphi_{N}^{WKB})'(\xi)=0$. Thus, Proposition~\ref{proposition_apriori} for $\hat{\varphi}_{0},\hat{\varphi}_{1}=0$ implies (note that $f_{N,\varepsilon}^{\pm}\in C(I)$ and $a\in W^{1,\infty}(I)$) the existence of some $C>0$ independent of $N$ and $\varepsilon$ such that
	\begin{align}\label{wkb_error_1}
		\lVert \varphi-\varphi_{N}^{WKB}\rVert_{L^{\infty}(I)}&\leq \frac{C}{\varepsilon}\lVert \alpha_{N,\varepsilon}\varphi_{N}^{-}f_{N,\varepsilon}^{-}+\beta_{N,\varepsilon}\varphi_{N}^{+}f_{N,\varepsilon}^{+}\rVert_{L^{2}(I)}\nonumber\\
		&\leq \frac{\widetilde{C}}{\varepsilon}\left(\lVert \alpha_{N,\varepsilon}\varphi_{N}^{-}\rVert_{L^{\infty}(I)}\lVert f_{N,\varepsilon}^{-}\rVert_{L^{\infty}(I)}+\lVert \beta_{N,\varepsilon}\varphi_{N}^{+}\rVert_{L^{\infty}(I)}\lVert f_{N,\varepsilon}^{+}\rVert_{L^{\infty}(I)}\right)\Comma
	\end{align}
	where $\widetilde{C}:=\sqrt{\eta-\xi}C$.
	Further, according to Corollary~\ref{corollary_Snk},
	\begin{align}\label{wkb_error_2}
		\lVert f_{N,\varepsilon}^{\pm}\rVert_{L^{\infty}(I)}&\leq 2\varepsilon^{N+1}\lVert S_{0}'\rVert_{L^{\infty}(I)}\lVert S_{N+1}'\rVert_{L^{\infty}(I)}+\sum_{n=2}^{N}\sum_{k=2+N-n}^{N}\varepsilon^{n+k}\lVert S_{n}'\rVert_{L^{\infty}(I)}\lVert S_{k}'\rVert_{L^{\infty}(I)}\nonumber\\
		&\leq \lVert S_{0}'\rVert_{L^{\infty}(G)}^{2}\Bigg(2\varepsilon^{N+1}K_{2}^{N+1}(N+1)^{N+1}+\sum_{n=2}^{N}\sum_{k=2+N-n}^{N}\varepsilon^{n+k}K_{2}^{n+k}n^{n}k^{k}\Bigg)\period
	\end{align}
	Estimate (\ref{wkb_error}) now follows from (\ref{wkb_error_1})-(\ref{wkb_error_2}) by applying Lemma~\ref{lemma_alpha_N_eps_phi_est}. This concludes the proof.
	\end{proof}
	\begin{Remark}\label{remark_wkb_error}
	%
	%
	%
        As a consequence of Theorem~\ref{theorem_wkb_error}, we have that
        \begin{align}
            \lVert\varphi-\varphi_{N}^{WKB}\rVert_{L^{\infty}(I)}=\mathcal{O}(\varepsilon^{N})\Comma\quad \varepsilon\to 0\period
        \end{align}
	\end{Remark}
\subsection{Refined error estimate incorporating quadrature errors}
Theorem~\ref{theorem_wkb_error} yields an explicit (w.r.t.\ $\varepsilon$ and $N$) error estimate for the WKB approximation (\ref{general_wkb_solution}). However, in practice one cannot expect to be able to compute (\ref{general_wkb_solution}) exactly. Indeed, even though for a given function $a$ one can compute the derivatives $(S_{n}^{\pm})'$ exactly through (\ref{S0})-(\ref{Sn}), one still has to deal with the integrals $\int_{\xi}^{x}(S_{n}^{\pm})'\,\mathrm{d}\tau$ in (\ref{Sn_definition}) in order to compute the functions $S_{n}^{\pm}$. For a detailed description of the method we use to approximate these integrals, we refer to Section~\ref{spectral_methods}. 

For now, let us assume we are given numerical approximations $\widetilde{S}_{n}^{\pm}$, $n\in\mathbb{N}_{0}$, of the functions $S_{n}^{\pm}$ that satisfy $\lVert \widetilde{S}_{n}^{\pm}-S_{n}^{\pm}\rVert_{L^{\infty}(I)}\leq e_{n}$ with positive constants $e_{n}$. We then define the corresponding ``perturbed'' WKB approximation as
\begin{align}
	\widetilde{\varphi}_{N}^{WKB}:=\alpha_{N,\varepsilon}\exp\left(\sum_{n=0}^{N}\varepsilon^{n-1}\widetilde{S}_{n}^{-}\right)+\beta_{N,\varepsilon}\exp\left(\sum_{n=0}^{N}\varepsilon^{n-1}\widetilde{S}_{n}^{+}\right)\period\label{general_wkb_solution_perturbed}
\end{align}
Notice that we use here the exact constants $\alpha_{N,\varepsilon}$ and $\beta_{N,\varepsilon}$ as given by formulas (\ref{alpha_N_eps})-(\ref{beta_N_eps}) (since the values $(S_{n}^{\pm})'(\xi)$ are exactly known from (\ref{S0})-(\ref{Sn})).

We are now interested in an error estimate for the perturbed WKB approximation (\ref{general_wkb_solution_perturbed}).
%
%
%
Such an estimate is provided by the following theorem:
\begin{Theorem}\label{theorem_wkb_error_approximate_Sn}
	Let Hypothesis~\ref{hypothesis_A} be satisfied and let $\varphi\in C^{2}(I)$ be the solution of IVP (\ref{schroedinger_eq}). Further assume $\widetilde{S}_{2n}^{\pm}(x)\in \ii\mathbb{R}$, $n\in\mathbb{N}_{0}$, for any $x\in I$. Then, there exist constants $\varepsilon_{0}\in(0,1)$ and $C>0$ independent of $N$ and $\varepsilon$ such that it holds for $\varepsilon\in(0,\varepsilon_{0}]$:
	\begin{align}\label{wkb_error_including_approx}
		\lVert \varphi-\widetilde{\varphi}_{N}^{WKB}\rVert_{L^{\infty}(I)}&\leq \exp\left((\eta-\xi)\lVert S_{0}'\rVert_{L^{\infty}(G)}\sum_{n=0}^{\lfloor \frac{N-1}{2}\rfloor}\varepsilon^{2n}K_{2}^{2n+1}(2n+1)^{2n+1}\right)\nonumber\\
		&\quad\times\Bigg[C\lVert S_{0}'\rVert_{L^{\infty}(G)}^{2}\left(|\varphi_{0}|\,\lVert S_{0}'\rVert_{L^{\infty}(G)}\sum_{n=0}^{N}\varepsilon^{n}K_{2}^{n}n^{n}+|\varphi_{1}|\right)\nonumber\\
		&\quad\quad\quad\times\Bigg(\varepsilon^{N}K_{2}^{N+1}(N+1)^{N+1}+\sum_{n=2}^{N}\sum_{k=2+N-n}^{N}\varepsilon^{n+k-1}K_{2}^{n+k}n^{n}k^{k}\Bigg)\nonumber\\
        &\quad\quad\quad+\left(|\alpha_{N,\varepsilon}|+|\beta_{N,\varepsilon}|\right)\left(\sum_{n=0}^{N}\varepsilon^{n-1}e_{n}\right)\exp\left(\sum_{n=0}^{\lfloor \frac{N-1}{2}\rfloor}\varepsilon^{2n}e_{2n+1}\right)\Bigg]\period
	\end{align}
    For $N=0$ the sums in the exponential functions drop, and for $N<2$ the double sum is dropped.
\end{Theorem}
\begin{proof}
\anton{We have that}
	\begin{align}\label{error_dreiecks}
		\lVert \varphi-\widetilde{\varphi}_{N}^{WKB}\rVert_{L^{\infty}(I)}\leq \lVert \varphi-\varphi_{N}^{WKB}\rVert_{L^{\infty}(I)}+\lVert \varphi_{N}^{WKB}-\widetilde{\varphi}_{N}^{WKB}\rVert_{L^{\infty}(I)}\period
	\end{align}
	Now, the first term in (\ref{error_dreiecks}) can be estimated using Theorem~\ref{theorem_wkb_error} and enforces the restriction $\varepsilon\in (0,\varepsilon_{0}]$. For the second term we estimate
	%
    \begin{align}\label{phi_wkb_minus_phi_wkb_perturbed}
		\lVert \varphi_{N}^{WKB}-\widetilde{\varphi}_{N}^{WKB}\rVert_{L^{\infty}(I)}&\leq |\alpha_{N,\varepsilon}|\left\lVert \exp\left(\sum_{n=0}^{N}\varepsilon^{n-1}S_{n}^{-}\right)-\exp\left(\sum_{n=0}^{N}\varepsilon^{n-1}\widetilde{S}_{n}^{-}\right)\right\rVert_{L^{\infty}(I)}\nonumber\\
		&\quad+|\beta_{N,\varepsilon}|\left\lVert \exp\left(\sum_{n=0}^{N}\varepsilon^{n-1}S_{n}^{+}\right)-\exp\left(\sum_{n=0}^{N}\varepsilon^{n-1}\widetilde{S}_{n}^{+}\right)\right\rVert_{L^{\infty}(I)}\period
	\end{align}
    Let us introduce the abbreviation $\Delta S_{n}^{\pm}:=\widetilde{S}_{n}^{\pm}-S_{n}^{\pm}$ and estimate
	\begin{align}\label{exp_quad_error}
		&\left\lVert \exp\left(\sum_{n=0}^{N}\varepsilon^{n-1}S_{n}^{\pm}\right)-\exp\left(\sum_{n=0}^{N}\varepsilon^{n-1}\widetilde{S}_{n}^{\pm}\right)\right\rVert_{L^{\infty}(I)}\nonumber\\
		& = \left\lVert\int_{0}^{1} \frac{\mathrm{d}}{\mathrm{d}t}\exp\left(\sum_{n=0}^{N}\varepsilon^{n-1}\left(S_{n}^{\pm}+t\Delta S_{n}^{\pm}\right)\right)\,\mathrm{d}t\right\rVert_{L^{\infty}(I)}\nonumber\\
		&\leq\left(\int_{0}^{1} \left\lVert\exp\left(\sum_{n=0}^{N}\varepsilon^{n-1}\left(S_{n}^{\pm}+t\Delta S_{n}^{\pm}\right)\right)\right\rVert_{L^{\infty}(I)}\,\mathrm{d}t\right)\left(\sum_{n=0}^{N}\varepsilon^{n-1}\lVert\Delta S_{n}^{\pm}\rVert_{L^{\infty}(I)}\right)\nonumber\\
		&\leq \left(\int_{0}^{1} \exp\left(\sum_{n=0}^{\lfloor \frac{N-1}{2}\rfloor}\varepsilon^{2n}\left(\lVert S_{2n+1}^{\pm}\rVert_{L^{\infty}(I)}+t\lVert\Delta S_{2n+1}^{\pm}\rVert_{L^{\infty}(I)}\right)\right)\,\mathrm{d}t\right)\left(\sum_{n=0}^{N}\varepsilon^{n-1}\lVert\Delta S_{n}^{\pm}\rVert_{L^{\infty}(I)}\right)\nonumber\\
		&\leq \exp\left(\sum_{n=0}^{\lfloor \frac{N-1}{2}\rfloor}\varepsilon^{2n}\left(\lVert S_{2n+1}^{\pm}\rVert_{L^{\infty}(I)}+\lVert\Delta S_{2n+1}^{\pm}\rVert_{L^{\infty}(I)}\right)\right)\left(\sum_{n=0}^{N}\varepsilon^{n-1}\lVert\Delta S_{n}^{\pm}\rVert_{L^{\infty}(I)}\right)\nonumber\\
		&\leq \exp\left((\eta-\xi)\lVert S_{0}'\rVert_{L^{\infty}(G)}\sum_{n=0}^{\lfloor \frac{N-1}{2}\rfloor}\varepsilon^{2n}K_{2}^{2n+1}(2n+1)^{2n+1}\right)\exp\left(\sum_{n=0}^{\lfloor\frac{N-1}{2}\rfloor}\varepsilon^{2n}e_{2n+1}\right)\left(\sum_{n=0}^{N}\varepsilon^{n-1}e_{n}\right)\Comma
	\end{align}
	where we used in the third step that $S_{2n}^{\pm}(x)+t\Delta S_{2n}^{\pm}(x)\in \ii\mathbb{R}$ for every $t\in[0,1]$ and $x\in I$, which is a direct consequence of (\ref{proposition_S_and_S_tilde_implication1}) and the assumption $\widetilde{S}_{2n}^{\pm}(x)\in \ii\mathbb{R}$. Moreover, in the last step we used Corollary~\ref{corollary_Snk}. The claim now follows by 
    combining (\ref{error_dreiecks})-(\ref{exp_quad_error}).
\end{proof}
Let us compare the extended error estimate (\ref{wkb_error_including_approx}) with (\ref{wkb_error}). The new (additional) second term inside the square brackets in (\ref{wkb_error_including_approx}) is due to the perturbed functions $\widetilde{S}_{n}^{\pm}$ and includes the approximation error bounds $e_{n}$. In particular, the factor $\sum_{n=0}^{N}\varepsilon^{n-1}e_{n}$ is rather unfavorable, as it is of order $\mathcal{O}(\varepsilon^{-1})$, as $\varepsilon\to 0$. We note that the appearance of this $\mathcal{O}(\varepsilon^{-1})$-term in estimate (\ref{wkb_error_including_approx}) is strongly related to the appearance of the $\mathcal{O}(\varepsilon^{-1})$-terms in \cite[Theorem 3.1]{Arnold2011WKBBasedSF}, \cite[Theorem 3.2]{Arnold2022WKBmethodFT} and \cite[Eq.\ (35)]{Jahnke2003NumericalIF}. There it implied an upper step size limit $h\leq \bar{h}(\varepsilon)=\varepsilon^{\gamma}$ with some $\gamma\in(0,1)$. Similarly, it would require here some $\varepsilon$-dependent upper bound on the quadrature error $e_{0}$ of $\widetilde{S}_{0}^{\pm}$ \anton{in order to compensate at least the $\mathcal{O}(\varepsilon^{-1})$-error term. In practice this will necessitate to use some finer grid for computing $\widetilde{S}_{0}^{\pm}$, as $\varepsilon$ decreases.} 
We specify this observation in the following remark. 
\begin{Remark}\label{remark_wkb_error_perturbed}
    It is evident from (\ref{alpha_N_eps_est1}) that $\alpha_{N,\varepsilon}=\mathcal{O}(1)$, $\varepsilon\to 0$. The same holds for $\beta_{N,\varepsilon}$. Hence, we see from (\ref{wkb_error_including_approx}) that
	\begin{align}\label{wkb_error_including_approx_order}
		\lVert \varphi-\widetilde{\varphi}_{N}^{WKB}\rVert_{L^{\infty}(I)} = \mathcal{O}(\varepsilon^{N})+\sum_{n=0}^{N}\mathcal{O}(\varepsilon^{n-1})e_{n}\Comma\quad \varepsilon\to 0\period
	\end{align}
	Thus, asymptotically, as $\varepsilon\to 0$, the approximation error of $\widetilde{S}_{0}^{\pm}$ has the biggest impact on the overall error since it is multiplied by a factor $\mathcal{O}(\varepsilon^{-1})$. In order to recover an overall $\mathcal{O}(\varepsilon^{N})$ error behavior, as in Theorem~\ref{theorem_wkb_error}, one should hence aim for highly accurate approximations of the functions $S_{n}^{\pm}$, with an $\varepsilon$-dependent error order of at most $e_{n}=\mathcal{O}(\varepsilon^{N-n+1})$.  
\end{Remark}
%
\section{Computation of the WKB approximation}\label{section_computation_of_the_WKB_approximation}
In this section we present the methods we use to compute the (perturbed) WKB approximation (\ref{general_wkb_solution}), (\ref{general_wkb_solution_perturbed}). This process can be divided into two steps. First, the computation of the functions $S_{n}$. Second, an adequate truncation of the asymptotic series (\ref{wkb_S}).
\subsection{Computation of the functions $S_{n}$}
\label{spectral_methods}
The computation of the functions $S_{n}$ relies on recurrence relation (\ref{S0})-(\ref{Sn}) as well as on definition (\ref{Sn_definition}). Since the latter involves the evaluation of an integral, one cannot expect to be able to compute $S_{n}$ exactly, in general. Consequently, we will instead compute approximations $\widetilde{S}_{n}\approx S_{n}=\int_{\xi}^{x}S_{n}'\,\mathrm{d}\tau$ which satisfy the assumption $\widetilde{S}_{2n}(x)\in \ii\mathbb{R}$, $n\in\mathbb{N}_{0}$, such that the resulting error for the corresponding perturbed WKB approximation can be controlled by Theorem~\ref{theorem_wkb_error_approximate_Sn}.

As the first step, we compute the derivatives $S_{n}'$ through (\ref{S0})-(\ref{Sn}) exactly, employing symbolic computations\footnote{As an alternative to (\ref{S0})-(\ref{Sn}), in \cite{Robnik2000SomePO} the authors established an almost explicit formula for the derivatives $S_{n}'$, depending on $a$ and its derivatives $a',\dots,a^{(n)}$. Although not used here, this approach may prove advantageous with regard to the computational time.}. Secondly, we employ a highly accurate quadrature for approximating the integral in (\ref{Sn_definition}). For this, we use the well-known Clenshaw-Curtis algorithm \cite{Clenshaw1960AMF}, which we shall briefly explain in the following.

The basic idea of Clenshaw-Curtis quadrature is to expand the integrand $f$ in terms of Chebyshev polynomials, the integrals of which are known. More precisely, one considers a truncated Chebyshev series for the integrand, i.e.,\
%
$f(l)\approx \sum_{r=0}^{M}a_r T_r(l)$, $l\in[-1,1]$,
where $T_r(l)=\cos(r\arccos(l))$, $r\in\mathbb{N}_{0}$, are the Chebyshev polynomials. Here, the spectral coefficients $a_r$ are determined with a collocation method at the Chebyshev collocation points $l_k=\cos(k\pi/M)$, $k=0,\dots,M$, by solving the $M+1$ equations
%
$f(l_k)=\sum_{r=0}^{M}a_{r}\cos\left(\frac{rk\pi}{M}\right)$
for the $a_r$, $r=0,\dots,M$. Therefore, the spectral coefficients can be computed by the discrete cosine transformation (DCT) of the function $f$ sampled at the collocations points. We note that the DCT is related to the discrete Fourier transform and can be computed efficiently using the fast Fourier transform algorithm after some preprocessing (e.g., see \cite[Chap.\ 8]{Trefethen2000SpectralMI}).

Then, the antiderivative of $f$ can be approximated again by a Chebyshev sum,
\begin{align}\label{clencurt_integral_br}
    \int_{-1}^{l}f(\tau)\,\mathrm{d}\tau \approx \sum_{r=0}^M b_{r} T_{r}(l)\Comma  
\end{align}
where the coefficients $b_{r}$ are related to the $a_{r}$, see \cite{Clenshaw1960AMF} for the detailed formulas. In \cite{Chawla1968ErrorEF} it was shown that the Clenshaw-Curtis method approximates integrals of analytic functions with spectral accuracy, i.e., the numerical error  \anton{decreases} exponentially with the number of modes $M$.

An integration over the interval $x\in[\xi,\eta]$ is realized by mapping $x=\eta(1+l)/2+\xi(1-l)/2$, $l\in[-1,1]$ to the interval $[-1,1]$. Thus, by sampling the derivatives $S_{n}'$ at the transformed Chebyshev points $x_{k}$, $k=0,\dots,M$ in the interval $I=[\xi,\eta]$, we obtain the approximations $\widetilde{S}_{n}(x_{k})\approx S_{n}(x_{k})$. Notably, the coefficients $b_{r}$ are such that the r.h.s.\ of (\ref{clencurt_integral_br}) vanishes at $l=-1$, implying $\widetilde{S}_{n}(\xi)=0$. Hence, the perturbed WKB approximation (\ref{general_wkb_solution_perturbed}) satisfies the first initial condition in (\ref{schroedinger_eq}), namely, $\widetilde{\varphi}_{N}^{WKB}(\xi)=\alpha_{N,\varepsilon}+\beta_{N,\varepsilon}=\varphi_{0}$. Finally, it is worth mentioning that when employing the Clenshaw-Curtis algorithm for the integrals in (\ref{Sn_definition}), it follows that $\widetilde{S}_{2n}(x_{k})\in \ii\mathbb{R}$. As a consequence, the error of the corresponding perturbed WKB approximation (\ref{general_wkb_solution_perturbed}) can be controlled with the aid of Theorem~\ref{theorem_wkb_error_approximate_Sn}.

\medskip
We note that an alternative and efficient way of approximating the functions $S_{n}$ can be realized without the need for symbolic computation of the derivatives $S_{n}'$. Indeed, one can instead employ a spectral method to perform the differentiation of the predecessor $S_{n-1}'$ in the recursion (\ref{Sn}). For instance, by using the $(M+1)\times(M+1)$ Chebyshev differentiation matrices $\mathbf{D}_{M}$ as described in \cite[Chap.\ 6]{Trefethen2000SpectralMI}, one can efficiently approximate the derivative of a function at Chebyshev grid points $l_{k}\in[-1,1]$, $k=0,1,\dots,M$. Thus, to approximate the derivative of a function sampled at transformed Chebyshev points $x_{k}\in[\xi,\eta]$, it is necessary to use the scaled matrix $\widetilde{\mathbf{D}}_{M}:=\frac{2}{\eta-\xi}\mathbf{D}_{M}$.
Following recurrence relation (\ref{S0})-(\ref{Sn}), we can therefore approximate the derivatives $S_{n}'$ sampled at Chebyshev points $x_{k}$ through the following pointwise definition on the grid:
\begin{align}
	\bar{S}_{1}^{\prime}(x_{k})&:=-\frac{\sum_{l=0}^{M}(\widetilde{\mathbf{D}}_{M})_{k+1,l+1}S_{0}^{\prime}(x_{l})}{2S_{0}^{\prime}(x_{k})}\Comma\label{S_1_approx}\\
	\bar{S}_{n}^{\prime}(x_{k})&:=-\frac{\sum_{j=1}^{n-1}\bar{S}_{j}^{\prime}(x_{k})\bar{S}_{n-j}^{\prime}(x_{k})+\sum_{l=0}^{M}(\widetilde{\mathbf{D}}_{M})_{k+1,l+1}\bar{S}_{n-1}^{\prime}(x_{l})}{2S_{0}^{\prime}(x_{k})}\Comma\quad  n\geq 2\Comma\label{S_n_approx}
\end{align}
for $k=0,\dots,M$. One then obtains approximations $\widetilde{S}_{n}(x_{k})\approx S_{n}(x_{k})$ by employing the Clenshaw-Curtis algorithm using the approximations $\bar{S}_{n}^{\prime}(x_{k})\approx S_{n}^{\prime}(x_{k})$, $k=0,\dots,M$.

However, note that approximating $S_{n}'$ using (\ref{S_1_approx})-(\ref{S_n_approx}) can lead to a rapid accumulation of errors, as repeated numerical differentiation is intrinsically unstable. The reason for this behavior lies in the ill-conditioned Chebyshev differentiation matrices $\mathbf{D}_{M}$. It is known that the condition number of these matrices is of order $\mathcal{O}(M^2)$ (e.g., see \cite{Breuer1992OnTE,Funaro1987APM}). In a finite precision approach this leads to a big loss, which means that 
in each application of the recurrence relation 
\anton{we lose a finite amount of accuracy in the computation of} the $S_n'$ \anton{ (see Figure \ref{fig:diff-matrix} for two examples).} Consequently, \markus{it can be recommended to employ} this approach \markus{only} for small values of $N$.
\begin{figure}
	\centering
	\includegraphics[scale=0.75]{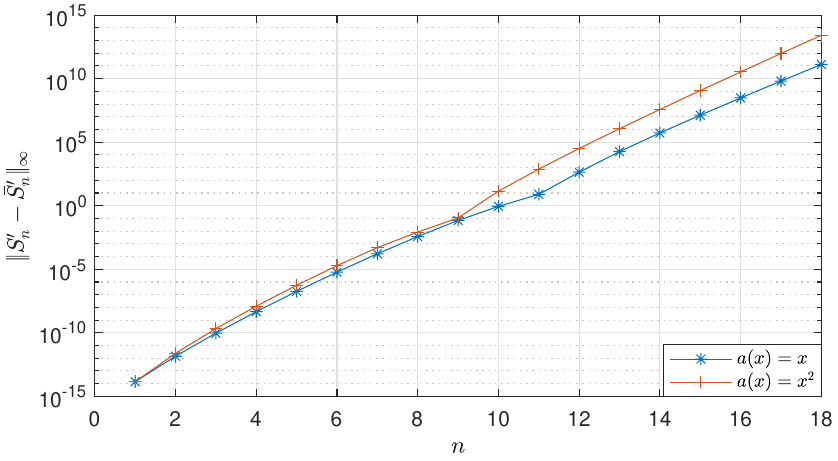}
	\caption{\anton{$L^{\infty}(I)$-norm of the error of the approximation $\bar{S}_{n}^{\prime}$ for the examples $a(x)=x$ and $a(x)=x^{2}$ on the interval $I=[1,2]$. Here, we set $M=20$.}}
	\label{fig:diff-matrix}
\end{figure}
\subsection{Truncation of the WKB series}\label{truncation_strategy}
When truncating the asymptotic series
\begin{align}
    f\sim\sum_{n=0}^{\infty}\varepsilon^{n}f_{n}\Comma \quad \varepsilon\to 0
\end{align}
after some finite order $N$, one would like to analyze the difference $f-\sum_{n=0}^{N}\varepsilon^{n}f_{n}$. But since the function $S$ in (\ref{WKB-ansatz})-(\ref{wkb_S}) remains unknown, we shall investigate the numerical error of the WKB approximation, 
%
%
as started in Section~\ref{section_error_analysis}. 
%
%

Recall that for a fixed $N\geq 0$, Theorem~\ref{theorem_wkb_error} guarantees that $\lVert \varphi-\varphi_{N}^{WKB}\rVert_{L^{\infty}(I)}=\mathcal{O}(\varepsilon^{N})$ as $\varepsilon \to 0$, see also Remark~\ref{remark_wkb_error}. In practical applications, however, the situation is exactly the opposite, namely, the small parameter $\varepsilon$ is fixed and $N$ can be chosen freely. Note also that just including more terms into the series (\ref{wkb_S}) does not necessarily reduce the error of the WKB approximation, simply since the asymptotic series is typically divergent. Hence the question arises which choice of $N$ will minimize $\lVert \varphi-\varphi_{N}^{WKB}\rVert_{L^{\infty}(I)}$, often referred to as \textit{optimal truncation}. In this context, we denote by $N_{opt}=N_{opt}(\varepsilon):=\operatorname{argmin}_{N\in \mathbb{N}_{0}}\lVert \varphi-\varphi_{N}^{WKB}\rVert_{L^{\infty}(I)}$ the \textit{optimal truncation order}. In general, an \textit{optimally truncated} asymptotic series is sometimes referred to as \textit{superasymptotics} (e.g., see \cite{Boyd1999TheDI}). The corresponding error of an optimally truncated series is then typically of the form $\sim \exp(-c/\varepsilon)$, as $\varepsilon\to 0$, with some constant $c>0$. 

In practice, a useful heuristic for finding the optimal truncation order 
for a fixed $\varepsilon$ is given in \cite{Boyd1999TheDI}. It suggests that 
it can be obtained by truncating the asymptotic series before its smallest term
.
In our case, we would hence have to find the minimizer $N_{heu}=N_{heu}(\varepsilon)$ of $n\mapsto\varepsilon^{n}\lVert S_{n+1} \rVert_{L^{\infty}(I)}$. This can either be found by ``brute force'', comparing the size of each term up to some prescribed maximal order $N_{max}$, or by utilizing Corollary~\ref{corollary_Snk} to (roughly) predict $N_{heu}$. Indeed, for any $N\in\mathbb{N}_{0}$, estimate (\ref{Sn_est}) implies
\begin{align}\label{Sn_est_epsi}
	\varepsilon^{N}\lVert S_{N+1}\rVert_{L^{\infty}(I)}\leq (\eta-\xi)\lVert S_{0}'\rVert_{L^{\infty}(G)}\varepsilon^{N}K_{2}^{N+1}(N+1)^{N+1}\period
\end{align}
Treating $N$ as a continuous variable \anton{for the moment}, we find the minimum of
\begin{align}
	g(N):=\ln\left(\varepsilon^{N}K_{2}^{N+1}(N+1)^{N+1}\right)
\end{align}
at
\begin{align}\label{hat_N_heu}
	\widehat{N}_{heu}=\widehat{N}_{heu}(\varepsilon)=\frac{1}{\mathrm{e}K_{2}\varepsilon}-1\period
\end{align}	
Hence, the minimum of the right-hand side of (\ref{Sn_est_epsi}) is
\begin{align}
	(\eta-\xi)\lVert S_{0}'\rVert_{L^{\infty}(G)}\exp\left(g(\widehat{N}_{heu})\right)=\frac{(\eta-\xi)\lVert S_{0}'\rVert_{L^{\infty}(G)}}{\varepsilon}\exp\left(-\frac{1}{\mathrm{e}K_2\varepsilon}\right)\period
\end{align}
So, the first term of the remainder of the asymptotic series appearing in the WKB-ansatz (\ref{WKB-ansatz})-(\ref{wkb_S}), truncated at the nearest integer value to $\widehat{N}_{heu}$, is exponentially small w.r.t.\ $\varepsilon$. Recalling that the term $\exp(g(N))=\varepsilon^{N}K_{2}^{N+1}(N+1)^{N+1}$ also appears in estimate (\ref{wkb_error}), we therefore might also expect the error $\lVert \varphi-\varphi_{N}^{WKB}\rVert_{L^{\infty}(I)}$ to be exponentially small w.r.t.\ $\varepsilon$, if $N$ is chosen adequately. Indeed, this is guaranteed by the following corollary of Theorem~\ref{theorem_wkb_error}.
\begin{Corollary}\label{corollary_exp_small_error}
    Let Hypothesis~\ref{hypothesis_A} be satisfied and let $\varphi\in C^{2}(I)$ be the solution of IVP (\ref{schroedinger_eq}). Then there exist $\tilde{\varepsilon}_{0}\in(0,1)$ and $N=N(\varepsilon)\in \mathbb{N}$ such that it holds for $\varepsilon\in(0,\tilde{\varepsilon}_{0}]$:
	%
    \begin{align}\label{exp_small_error}
		\lVert \varphi-\varphi_{N}^{WKB}\rVert_{L^{\infty}(I)} \leq 
        \anton{ C\exp\left(-\frac{r}{\varepsilon}\right) } \Comma
	\end{align}
    with constants $C,r>0$ independent of $\varepsilon$.
\end{Corollary}
\begin{proof}
    We prove estimate (\ref{exp_small_error}) by applying Theorem~\ref{theorem_wkb_error} for a specific choice of $N=N(\varepsilon)$. First, choose $0<\tilde{\varepsilon}_{0}<\min(\varepsilon_{0},\frac{1}{K_{2}})$ with $\varepsilon_{0}\in(0,1)$ being the constant from Theorem~\ref{theorem_wkb_error} and $K_{2}$ from Corollary~\ref{corollary_Snk}. Then there exists some constant $c\in[\ee K_{2}\tilde{\varepsilon}_{0},\ee)$ implying that $N:=\lfloor \frac{c}{\ee K_{2} \varepsilon}\rfloor-1\geq 0$ for any $\varepsilon\in(0,\tilde{\varepsilon}_{0}]$. 
    The idea is now to majorize, for this choice of $N$, several sums in (\ref{wkb_error}) by convergent geometric series. First, we have
	\begin{align}\label{geom_series_1}
		\sum_{n=0}^{N}(\varepsilon K_{2} n)^{n}\leq \sum_{n=0}^{N}(\varepsilon K_{2} (N+1))^{n}\leq \sum_{n=0}^{\infty}\left(\frac{c}{\ee}\right)^{n}=\frac{1}{1-\frac{c}{\ee}}\Comma
	\end{align}
	where we used $\varepsilon K_{2} (N+1)\leq\frac{c}{\ee}$. Similarly, we get
	\begin{align}\label{geom_series_2}
		\sum_{n=0}^{\lfloor \frac{N-1}{2}\rfloor}\varepsilon^{2n} (K_{2} (2n+1))^{2n+1}&\leq K_{2}\sum_{n=0}^{\lfloor \frac{N-1}{2}\rfloor}(\varepsilon K_{2} (N+1))^{2n}(2n+1)\nonumber\\
		&\leq K_{2}\sum_{n=0}^{N}\left(\frac{c}{\ee}\right)^{n}(n+1)\nonumber\\
		&\leq K_{2}\left(\sum_{n=0}^{\infty}\left(\frac{c}{\ee}\right)^{n}n+\sum_{n=0}^{\infty}\left(\frac{c}{\ee}\right)^{n}\right)\nonumber\\
		&=K_{2}\frac{\ee^{2}}{(\ee-c)^{2}}\Comma
	\end{align}
	where we used the geometric series variant $\sum_{n=0}^{\infty}q^{n}n=\frac{q}{(1-q)^{2}}$ for any $q\in\mathbb{R}$ with $|q|<1$. At this point, Theorem~\ref{theorem_wkb_error} and (\ref{geom_series_1})-(\ref{geom_series_2}) imply for $\varepsilon\in(0,\tilde{\varepsilon}_{0}]$
	\begin{align}\label{optimal_error_interim_result}
		\lVert \varphi-\varphi_{N}^{WKB}\rVert_{L^{\infty}(I)} &\leq C\varepsilon^{N}K_{2}^{N+1}(N+1)^{N+1}+C\sum_{n=2}^{N}\sum_{k=2+N-n}^{N}\varepsilon^{n+k-1}K_{2}^{n+k}n^{n}k^{k}\Comma
	\end{align}
	where $C>0$ is some constant independent of $\varepsilon$.
	%
	Now, for the first term in (\ref{optimal_error_interim_result}) we have that
	\begin{align}\label{optimal_error_first_term}
		\varepsilon^{N}K_{2}^{N+1}(N+1)^{N+1}\leq \frac{1}{\varepsilon}\left(\frac{c}{\ee}\right)^{\left\lfloor \frac{c}{\ee K_{2}\varepsilon}\right\rfloor}\leq\frac{\ee}{c\varepsilon}\left(\frac{c}{\ee}\right)^{\frac{c}{\ee K_{2}\varepsilon}} = \frac{\ee}{c\varepsilon}\exp\left(-\frac{\anton{\tilde r}}{\varepsilon}\right)\Comma
	\end{align}
    with $\anton{\tilde r}:=\frac{c\ln(\ee/c)}{\ee K_{2}}>0$. Finally, the second term in (\ref{optimal_error_interim_result}) can be estimated as follows:
    \begin{align}\label{optimal_error_second_term}
        \sum_{n=2}^{N}\sum_{k=2+N-n}^{N}\varepsilon^{n+k-1}K_{2}^{n+k}n^{n}k^{k}&\leq \frac{1}{\varepsilon}\sum_{n=2}^{N}\sum_{k=2+N-n}^{N}(\varepsilon K_{2}N)^{n}(\varepsilon K_{2}N)^{k}\nonumber\\
        &=\frac{1}{\varepsilon}(\varepsilon K_{2}N)^{N+2}\sum_{n=0}^{N-2}\sum_{k=0}^{n}(\varepsilon K_{2}N)^{k}\nonumber\\
        &\leq \frac{1}{\varepsilon}(\varepsilon K_{2}N)^{N+2}\sum_{n=0}^{N-2}\sum_{k=0}^{\infty}\left(\frac{c}{\ee}\right)^{k}\nonumber\\
        &=\frac{1}{\varepsilon}(\varepsilon K_{2}N)^{N+2}\frac{N-1}{1-\frac{c}{\ee}}\nonumber\\
        &\leq \frac{K_{2}}{1-\frac{c}{\ee}}\left(\frac{c}{\ee K_{2}\varepsilon}\right)^{2}\left(\frac{c}{\ee}\right)^{\left\lfloor \frac{c}{\ee K_{2}\varepsilon}\right\rfloor}\nonumber\\
        &\leq \frac{c}{(\ee-c)K_{2}\varepsilon^{2}}\left(\frac{c}{\ee}\right)^{\frac{c}{\ee K_{2}\varepsilon}}\nonumber\\
        &= \frac{c}{(\ee-c)K_{2}\varepsilon^{2}}\exp\left(-\frac{\anton{\tilde r}}{\varepsilon}\right)\period
    \end{align}
    We observe that the r.h.s.\ of (\ref{optimal_error_first_term}) can be bounded by the r.h.s.\ of (\ref{optimal_error_second_term}) (up to a multiplicative constant) for $\varepsilon\in (0,\tilde{\varepsilon}_{0}]$. Thus, the claim follows \anton{with $r:=\tilde r/2$ and adapting $C$}.
    \end{proof}
\begin{Remark}\label{remark_optimal_error}
	We note that the specific value $N$ from the proof of Corollary~\ref{corollary_exp_small_error} is not necessarily equal to the optimal truncation order $N_{opt}$. However, as a consequence of Corollary~\ref{corollary_exp_small_error}, and by the definition of $N_{opt}$, we conclude that
	\begin{align}
		\lVert \varphi-\varphi_{N_{opt}}^{WKB}\rVert_{L^{\infty}(I)} =
    \anton{\mathcal{O}(\exp(-r/\varepsilon))}
    \Comma\quad \varepsilon\to 0\Comma\quad r>0\period
	\end{align}
\end{Remark}
Finally, we note that, apart from $N_{heu}$ and $\widehat{N}_{heu}$, another option for predicting the optimal truncation order $N_{opt}$ is to find the minimizer of error estimate (\ref{wkb_error}) (for $\varepsilon$ fixed), say $\widehat{N}_{opt}=\widehat{N}_{opt}(\varepsilon)$, although this rather complicated expression can only be minimized numerically by brute force.

At this point, it seems convenient to summarize the notations and meanings of the different mentioned truncation orders which aim to estimate $N_{opt}$ -- see Table~\ref{table:optimal_notation}. In the next section we will compare results for each truncation order from Table~\ref{table:optimal_notation}, since it is not clear a priori which of these orders provides the most accurate prediction of $N_{opt}$. Nonetheless, let us note that in our experiments $N_{opt}$, $\widehat{N}_{opt}$, and $N_{heu}$ can only be determined by brute force, while $\widehat{N}_{heu}$ is given explicitly by formula (\ref{hat_N_heu}).
\begin{table}
	\centering
	\begin{tabular}{ l l} 
        \toprule
		$N_{opt}$ & minimizer of $\lVert \varphi-\varphi_{N}^{WKB}\rVert_{L^{\infty}(I)}$ (optimal truncation order)\\
        $\widehat{N}_{opt}$ & minimizer of error estimate (\ref{wkb_error}) (prediction of $N_{opt}$)\\
		${N}_{heu}$ & minimizer of $\varepsilon^{N}\lVert S_{N+1}\rVert_{L^{\infty}(I)}$ (heuristic prediction of $N_{opt}$)\\  
		$\widehat{N}_{heu}$ & minimizer of the r.h.s.\ of (\ref{Sn_est_epsi}) (prediction of $N_{heu}$)\\ \bottomrule
	\end{tabular}
	\caption{Terminology for the different truncation orders mentioned in Section~\ref{section_computation_of_the_WKB_approximation}. The numbers $\widehat{N}_{opt}$ and $\widehat{N}_{heu}$ are predictions for $N_{opt}$ and ${N}_{heu}$ by means of (\ref{wkb_error}) and (\ref{Sn_est_epsi}), respectively. \anton{(For numerical values in two concrete examples see Figures \ref{plot:airy_N_opt_opt_err} and \ref{plot:third_example_N_opt}.)}}	
	\label{table:optimal_notation}
\end{table}
\section{Numerical simulations}\label{section_numerical_simulations}
%
In this section we present several numerical simulations to illustrate some of the theoretical results we derived in Section~\ref{section_error_analysis}. To this end, we will compute the (perturbed) WKB approximation as described in Section~\ref{spectral_methods}. That is, the functions $S_{n}'$ are pre-computed symbolically and are then integrated numerically using the Clenshaw-Curtis algorithm based on a Chebyshev grid with $M+1$ grid points, where $M$  will be specified later. All computations are carried out using \MATLAB version 9.13.0.2049777 (R2022b). Further, since we are dealing with very small errors for the WKB approximation, especially when investigating the optimal truncation order, we use the
Advanpix Multiprecision Computing Toolbox for \MATLAB \cite{Advanpix} with quadruple-precision to avoid roundoff errors.
\subsection{Example 1: Airy equation}
Consider the initial value problem
\begin{align}\label{eqn:Airy}
	\begin{cases}
		\varepsilon^{2}\varphi^{\prime\prime  }(x) + x \varphi(x) = 0 \Comma \quad x \in [1, 2] \Comma \\
		\varphi(1) = \Ai(-\frac{1}{\varepsilon^{2/3}}) + \ii \Bi(-\frac{1}{\varepsilon^{2/3}}) \Comma \\
		\varepsilon\varphi^{\prime}(1) = -\varepsilon^{1/3}\left(\Ai^{\prime}(-\frac{1}{\varepsilon^{2/3}}) + \ii \Bi^{\prime}(-\frac{1}{\varepsilon^{2/3}})\right) \Comma
	\end{cases}
\end{align}
where the exact solution is given by
\begin{align}\label{airy_exact_sol}
	\varphi_{exact}(x) = \Ai(-\frac{x}{\varepsilon^{2/3}}) + \ii \Bi(-\frac{x}{\varepsilon^{2/3}}) \period
\end{align}
Here, $\Ai$ and $\Bi$ denote the Airy functions of first and second kind, respectively (e.g., see \cite[Chap.\ 9]{Olver2010NISTHO}). Note that for this example, where $a(x)=x$, the derivatives $S_{n}^{\prime}$ are given by powers of $x$ (up to a constant factor). Hence, the functions $S_{n}^{\pm}$ can be computed exactly from (\ref{Sn_definition}); however, we shall use them here only as reference solutions for the approximations $\widetilde{S}_{n}^{\pm}$. Indeed, for a fixed number $M+1$ of Chebyshev grid points, we are then able to compute explicitly the approximation error $\lVert S_{n}^{\pm}-\widetilde{S}_{n}^{\pm}\rVert_{L^{\infty}(I)}=:e_{n}$. Since $\widetilde{S}_{n}^{\pm}$ is only available at the grid points, we actually compute the discrete analog of this norm.

The left panel of Figure~\ref{plot:airy_sol_plot} shows the real part of $\varphi_{exact}$ for the choice $\varepsilon=2^{-8}$, which illustrates well the highly oscillatory behavior of the solution.
\begin{figure}
	\centering
    \includegraphics[scale=0.75]{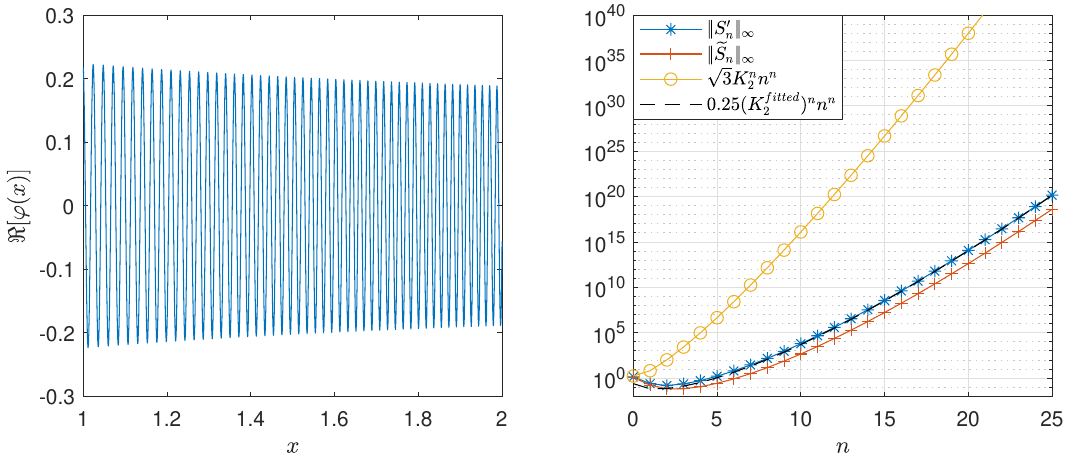}
	\caption{Left: Real part of the exact solution (\ref{airy_exact_sol}) of IVP (\ref{eqn:Airy}) for the choice $\varepsilon=2^{-8}$. Right: $L^{\infty}(I)$-norm of $S_{n}'$ and $\widetilde{S}_{n}$ as functions of $n$ for the example $a(x)=x$ on the interval $\anton{I=}[1,2]$.}
	\label{plot:airy_sol_plot}
\end{figure}
Let us first investigate numerically the result from Corollary~\ref{corollary_Snk}. For this, let us compute a constant $K_{2}$, as indicated by the proof of Corollary~\ref{corollary_Snk} and Remark~\ref{remark:K_2_trade_off}. Indeed, by using the minimization strategy from (\ref{G_minimization}), we find that (note that here $S_{0}'(z)=\pm\operatorname{i}\sqrt{z}$; \anton{$\delta_{opt}=\frac{\sqrt{97}-7}{6}\approx 0.4748$) }
%
\begin{align}
    K_{2}=\frac{\ee}{2\ee-1}
    \min_{0<\delta\leq 1}\anton{\frac{\sqrt{2+\delta}}{\delta(1-\delta)}=\frac{\ee}{2\ee-1}\frac{\sqrt{2+\delta_{opt}}}{\delta_{opt}(1-\delta_{opt})}\approx 3.8653 }
\end{align}
is a suitable constant within the context of Corollary~\ref{corollary_Snk}. In the right panel of Figure~\ref{plot:airy_sol_plot}, we present the $L^{\infty}(I)$-norms of the functions $S_{n}'$ and the approximations $\widetilde{S}_{n}$ when using $M=25$, along with the theoretical bound (\ref{Snk_est}) on $\lVert S_{n}'\rVert_{L^{\infty}(I)}$. We observe that the true norms consistently remain below the theoretical bound. Additionally, we include as a dashed line the theoretical bound (\ref{Snk_est}) when replacing $K_{2}$ and $\lVert S_{0}'\rVert_{L^{\infty}(G)}$ by the \anton{experimentally} fitted values $K_{2}^{fitted}=10/37\approx 0.27$ and $0.25$, respectively.\footnote{\anton{Dividing \eqref{Snk_est} by $n^n$ and taking the logarithm we used a linear approximation to obtain $K_2^{fitted}$.}} We observe very good agreement between the norms $\lVert S_{n}'\rVert_{L^{\infty}(I)}$ and the dashed line. This demonstrates well that, in the present example, the norms $\lVert S_{n}'\rVert_{L^{\infty}(I)}$ grow as Corollary~\ref{corollary_Snk} suggests, i.e., $\lVert S_{n}'\rVert_{L^{\infty}(I)}\sim C K_{2}^{n}n^{n}$ as $n\to \infty$, for some \anton{appropriate constants} $C,K_{2}>0$. In general, however, this may not be the case. We refer to Appendix~\ref{appendix_convergent} and Section~\ref{subsection_example2} for an example, where the functions $S_{n}'$ and $S_{n}$ even decay as $n\to \infty$.

Next, we investigate numerically the behavior of the WKB approximation error $\lVert \varphi-\widetilde{\varphi}_{N}^{WKB}\rVert_{L^{\infty}(I)}$ as a function of $\varepsilon$. We may compare the results with the error ``estimate'' (\ref{wkb_error_including_approx_order}). As a first test, we set $M=8$ to compute $\widetilde{\varphi}_{N}^{WKB}$. This results in an approximation error for $S_{n}$ of $e_{n}\approx 10^{-8}$, $n=0,\dots,4$. On the left of Figure~\ref{plot:airy_wkb_error_M=8} we plot for $N=0,\dots,4$ the error as a function of $\varepsilon$: For $N=2,3,4$ and small values of $\varepsilon$, the $\mathcal{O}(\varepsilon^{-1})e_{0}$-term is dominant. In contrast, for $N=0$ and $N=1$ this error term is not visible for the given range of $\varepsilon$-values so that the $\mathcal{O}(\varepsilon^{N})$-term is dominant. As a second test, we set again $M=8$, but now use in $\widetilde{\varphi}_{N}^{WKB}$ the exactly computed function $S_{0}$. The $\mathcal{O}(\varepsilon^{-1})e_{0}$-term from (\ref{wkb_error_including_approx_order}) is thus eliminated.
On the right of Figure~\ref{plot:airy_wkb_error_M=8} we show again the error $\lVert \varphi-\widetilde{\varphi}_{N}^{WKB}\rVert_{L^{\infty}(I)}$ as a function of $\varepsilon$: For $N=2,3,4$ and small $\varepsilon$-values, the $\mathcal{O}(\varepsilon^{0})e_{1}$-term, which is the next term in the sum in (\ref{wkb_error_including_approx_order}), now dominates. Indeed, the error curves show an almost constant value of approximately $2\cdot10^{-9}$ for small values of $\varepsilon$. For larger $\varepsilon$, the error curves behave like $\mathcal{O}(\varepsilon^{N})$. As a third test, we set $M=25$ and approximate again all functions $S_{n}$, $n=0,\dots,4$ (as in the first test). The corresponding approximation errors of $S_{n}$ are $e_{n}\approx 10^{-23}$, $n=0,\dots,4$.
On the left of Figure~\ref{plot:airy_wkb_error_M=25} we present the resulting WKB approximation errors. We observe that, on this scale, all $\mathcal{O}(\varepsilon^{n-1})e_{n}$-terms in the sum of (\ref{wkb_error_including_approx_order}) are essentially eliminated, since all the shown error curves behave like $\mathcal{O}(\varepsilon^{N})$. Overall, we observe very good agreement between the numerical results of each of the three tests and the statements from Theorem~\ref{theorem_wkb_error_approximate_Sn} and Remark~\ref{remark_wkb_error_perturbed}.

\begin{figure}
	\centering
    \includegraphics[scale=0.75]{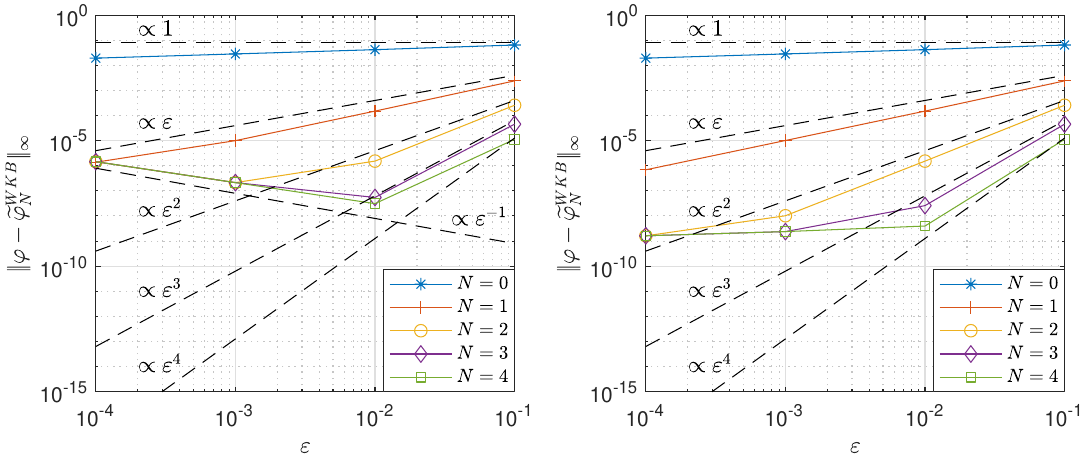}
	\caption{$L^{\infty}(I)$-norm of the error of the WKB approximation as a function of $\varepsilon$, for the IVP (\ref{eqn:Airy}) and several choices of $N$. Left: $M=8$. Right: $M=8$; using the \markus{exact} function $S_{0}$.
    }
	\label{plot:airy_wkb_error_M=8}
\end{figure}
%
%
\begin{figure}
	\centering
    \includegraphics[scale=0.75]{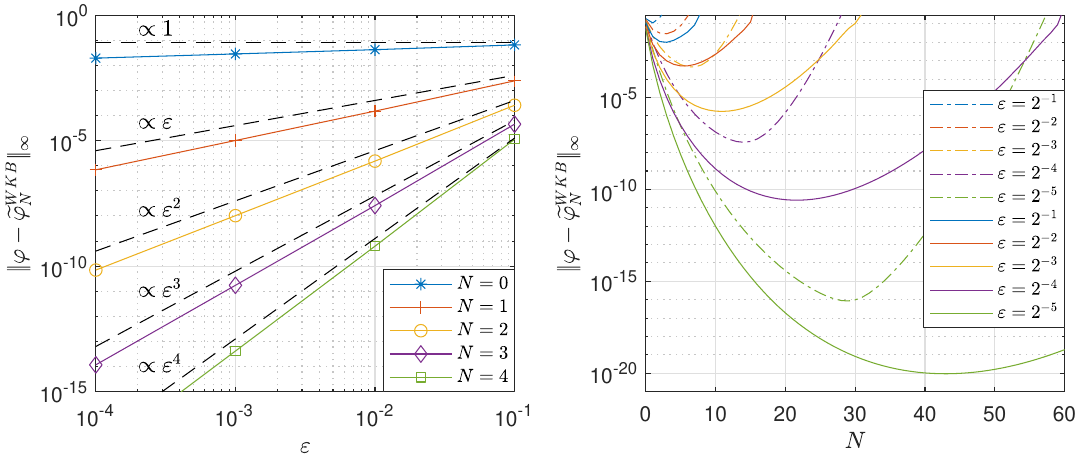}
	\caption{Left: $L^{\infty}(I)$-norm of the error of the WKB approximation as a function of $\varepsilon$ for the IVP (\ref{eqn:Airy}) and several choices of $N$.  Here, we set $M=25$. Right: $L^{\infty}(I)$-norm of the error of the WKB approximation as a function of $N$ for the IVP (\ref{eqn:Airy}) and several choices of $\varepsilon$. The dash-dotted lines correspond to the error estimate according to Theorem~\ref{theorem_wkb_error} and the solid lines correspond to the actual error of the WKB approximation. \anton{(In both plots the curves have the same top-down ordering as the legend.)} }
	\label{plot:airy_wkb_error_M=25}
\end{figure}
Next we investigate the error of the (perturbed) WKB approximation as a function of the truncation order $N$. For this, we set again $M=25$, yielding approximation errors of $S_{n}$ as $e_{n}\approx 10^{-23}$ for $n=0,1,\dots$. We may therefore neglect the errors caused by approximating the functions $S_{n}$. On the right of Figure~\ref{plot:airy_wkb_error_M=25} we plot the actual error $\lVert \varphi-\widetilde{\varphi}_{N}^{WKB}\rVert_{L^{\infty}(I)}$ and its error estimate (\ref{wkb_error}) while again using $K_{2}^{fitted}=10/37$, both as functions of $N$, for several $\varepsilon$-values. We observe that, even when using the fitted constant $K_{2}^{fitted}$, the ``optimal'' truncation order $\widehat{N}_{opt}$, as predicted by the estimate (\ref{wkb_error}), is smaller than $N_{opt}$ (determined as the $\operatorname{argmin}$ of the actual error curve). For instance, we have $\widehat{N}_{opt}(2^{-4})\approx 14<22\approx N_{opt}(2^{-4})$ and $\widehat{N}_{opt}(2^{-5})\approx 29<44\approx N_{opt}(2^{-5})$, respectively.
\anton{This is not a paradox, but it is implied in this example by the strong over-estimation (\ref{wkb_error}) of the error for large $N$.} 

In Figure~\ref{plot:airy_N_opt_opt_err} we plot on the left the optimal truncation order $N_{opt}(\varepsilon)$ as a function of $\varepsilon$ as well as its predictions $\widehat{N}_{opt}(\varepsilon)$, $N_{heu}(\varepsilon)$, and $\widehat{N}_{heu}(\varepsilon)$. The plot suggests that $N_{opt}$, $\widehat{N}_{opt}$, and $N_{heu}$ are \anton{all} proportional to $\varepsilon^{-1}$, as $\varepsilon \to 0$ (for $\widehat{N}_{heu}$ this is already evident from (\ref{hat_N_heu})). Further, for $\varepsilon=2^{-1},2^{-3},2^{-4},2^{-5}$ we observe that $N_{opt}=N_{heu}$. On the right of Figure~\ref{plot:airy_N_opt_opt_err} we plot the corresponding optimal error which is achieved by using $N_{opt}$ as well as error estimate (\ref{wkb_error}) when using $N=\widehat{N}_{opt}$, both as a function of 
\anton{$1/\varepsilon$. }
As indicated by the dashed line, the optimal error decays like 
\jannis{$\mathcal{O}(\exp(-r/\varepsilon))$, with $r\approx 1.36$}
being a fitted value, in good agreement with Remark~\ref{remark_optimal_error}.
%
%
%
\begin{figure}
	\centering
	\includegraphics[scale=0.75]{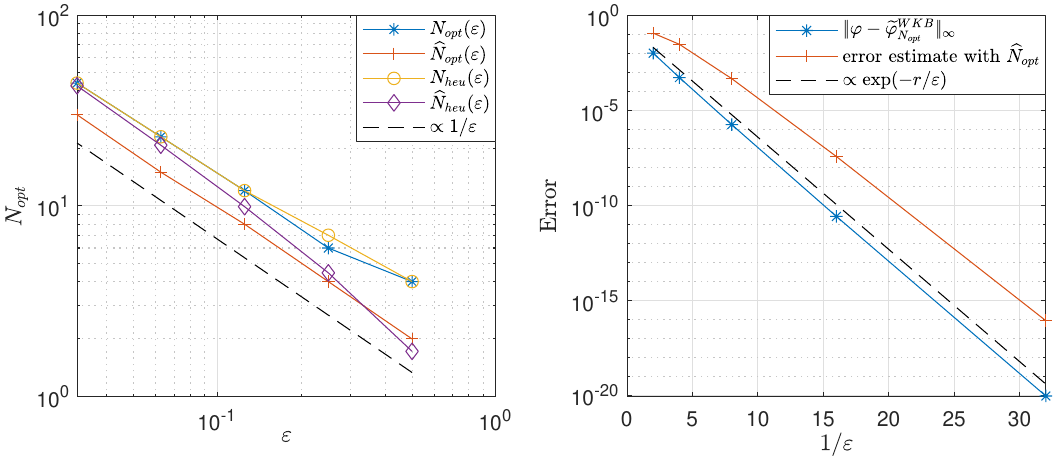}
    \caption{Left: The optimal truncation order $N_{opt}$ as well as the predicted ``optimal'' orders $\widehat{N}_{opt}$, $N_{heu}$, and $\widehat{N}_{heu}$ as functions of $\varepsilon$. The dashed line is proportional to $1/\varepsilon$. Right: The optimal error achieved by using $N_{opt}$ as well as error estimate (\ref{wkb_error}) when using $N=\widehat{N}_{opt}$, both as functions of \jannis{$1/\varepsilon$}. The dashed line is proportional \anton{to $\exp(-\frac{r}{\varepsilon})$ with $r\approx 1.36$. }}
	\label{plot:airy_N_opt_opt_err}
\end{figure}
\subsection{Example 2}
As our second example let us consider the initial value problem
\begin{align}\label{eqn:third_example}
	\begin{cases}
		\varepsilon^{2}\varphi^{\prime\prime  }(x) + \ee^{5x} \varphi(x) = 0 \Comma \quad x \in [0, 1] \Comma \\
		\varphi(0) = 1 \Comma \\
		\varepsilon\varphi^{\prime}(0) = 0 \Comma
	\end{cases}
\end{align}
where the exact solution is given by\footnote{We found the exact solution by using the Symbolic Math Toolbox of \MATLAB.}
\begin{align}\label{third_example_exact_sol}
    \varphi_{exact}(x)&=\frac{J_0\left(\frac{2}{5\varepsilon}\ee^{5x/2}\right)Y_1\left(\frac{2}{5\varepsilon}\right)-Y_0\left(\frac{2}{5\varepsilon}\ee^{5x/2}\right)J_1\left(\frac{2}{5\varepsilon}\right)}{J_0\left(\frac{2}{5\varepsilon}\right)Y_1\left(\frac{2}{5\varepsilon}\right)-J_1\left(\frac{2}{5\varepsilon}\right)Y_0\left(\frac{2}{5\varepsilon}\right)}\period
\end{align}
Here, $J_{\nu}$ and $Y_{\nu}$ denote the Bessel functions of first and second kind of order $\nu$, respectively (e.g., see \cite[Chap.\ 10]{Olver2010NISTHO}).

On the left of Figure~\ref{plot:third_example_sol_plot} the exact solution $\varphi_{exact}$ is plotted for the choice $\varepsilon=10^{-2}$. Throughout the whole interval, due to the fast growth of the function $a(x)=\exp(5x)$, the solution exhibits a rapid increase of its oscillatory behavior. 
\begin{figure}
	\centering
 \includegraphics[scale=0.75]{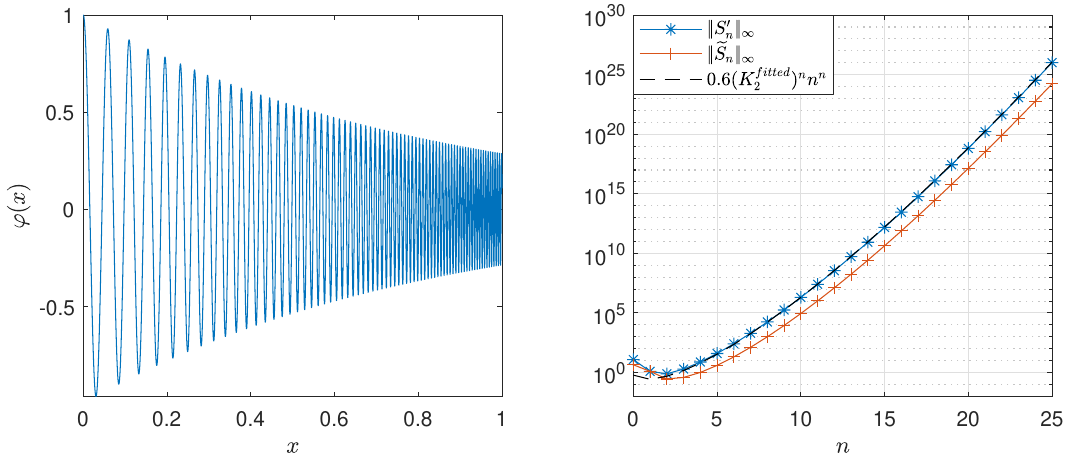}
	\caption{Left: Exact solution (\ref{third_example_exact_sol}) of IVP (\ref{eqn:third_example}) for the choice $\varepsilon=10^{-2}$. Right: $L^{\infty}(I)$-norm of $S_{n}'$ and $\widetilde{S}_{n}$ as functions of $n$ for the example $a(x)=\exp(5x)$ on the interval $I=[0,1]$.}
	\label{plot:third_example_sol_plot}
\end{figure}
Further, we plot on the right of Figure~\ref{plot:third_example_sol_plot} the $L^{\infty}(I)$-norms of the derivatives $S_{n}'$ and the approximations $\widetilde{S}_{n}$ when using $M=30$. As indicated by the dashed line, the smallest (fitted) constant $K_{2}$ such that estimate (\ref{Snk_est}) holds is $K_{2}^{fitted}\approx 9/20$ (here we also replaced $\lVert S_{0}' \rVert_{L^{\infty}(G)}$ in (\ref{Snk_est}) by the fitted value $0.6$).
In Figure~\ref{plot:third_example_err_vs_eps} we present the WKB approximation error $\lVert \varphi-\widetilde{\varphi}_{N}^{WKB}\rVert_{L^{\infty}(I)}$ as a function of $\varepsilon$ and may again compare the results with the error ``estimate'' (\ref{wkb_error_including_approx_order}). We observe that, on this scale, all $\mathcal{O}(\varepsilon^{n-1})e_{n}$-terms are essentially eliminated, since all the shown error curves behave like $\mathcal{O}(\varepsilon^{N})$. Overall, we observe very good agreement with the statements from Theorem~\ref{theorem_wkb_error_approximate_Sn} and Remark~\ref{remark_wkb_error_perturbed}. Finally, we plot in Figure~\ref{plot:third_example_N_opt} on the left the optimal truncation order $N_{opt}$ as well as its predictions $\widehat{N}_{opt}$, $N_{heu}$, and $\widehat{N}_{heu}$, as functions of $\varepsilon$. We find that $N_{opt}$ is proportional to $\varepsilon^{-1}$, as $\varepsilon\to 0$. On the right of Figure~\ref{plot:third_example_N_opt} we present the corresponding optimal error as well as error estimate (\ref{wkb_error}) when using $N=\widehat{N}_{opt}$, both as a function of 
\anton{$1/\varepsilon$. } 
As the dashed line indicates, the error decays like 
\jannis{$\mathcal{O}(\exp(-r/\varepsilon))$, with $r\approx 0.81$}
being a fitted value. This is in good agreement with Remark~\ref{remark_optimal_error}.
%
%
\begin{figure}
	\centering
    \includegraphics[scale=0.75]{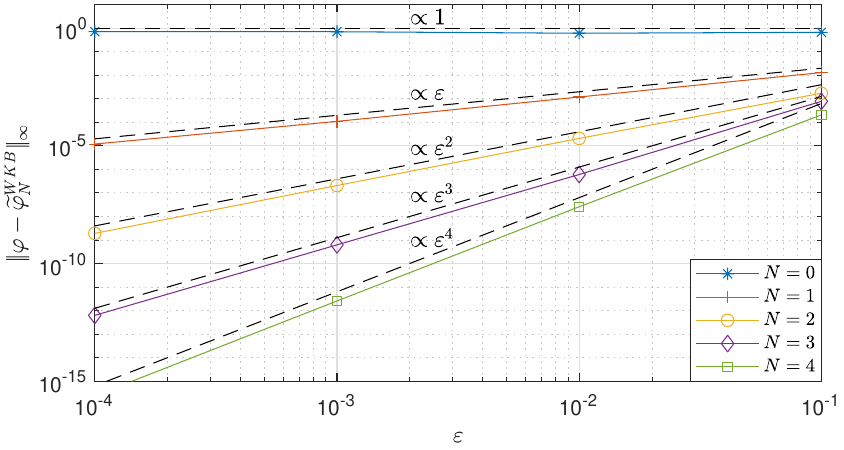}
	\caption{$L^{\infty}(I)$-norm of the error of the WKB approximation as a function of $\varepsilon$ for the IVP (\ref{eqn:third_example}) and several choices of $N$. Here, we set $M=30$. 
    }
	\label{plot:third_example_err_vs_eps}
\end{figure}
\begin{figure}
	\centering
    \includegraphics[scale=0.75]{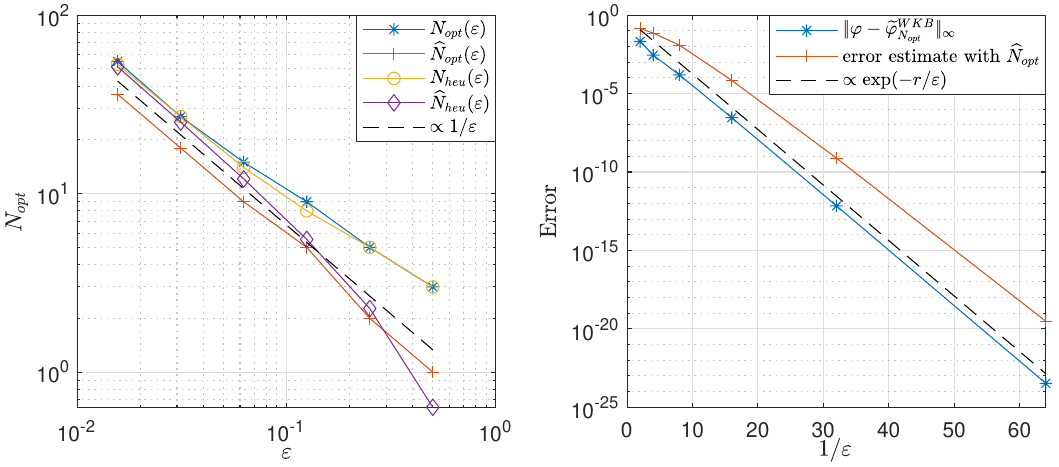}
	\caption{Left: The optimal truncation order $N_{opt}$ as well as the predicted ``optimal'' orders $\widehat{N}_{opt}$, $N_{heu}$, and $\widehat{N}_{heu}$ as functions of $\varepsilon$. The dashed line is proportional to $1/\varepsilon$. Right: The optimal error achieved by using $N_{opt}$ as well as error estimate (\ref{wkb_error}) when using $N=\widehat{N}_{opt}$, both as functions of 
    \anton{$1/\varepsilon$. The dashed line is proportional to $\exp(-\frac{r}{\varepsilon})$ with $r\approx0.81$.} }
	\label{plot:third_example_N_opt}
\end{figure}
\subsection{Example 3: Convergent WKB approximation}\label{subsection_example2}
As a final example, let us consider $a(x)=(1+x+x^{2})^{-2}$. We are interested in investigating the initial value problem
\begin{align}\label{eqn:second_example}
	\begin{cases}
		\varepsilon^{2}\varphi^{\prime\prime  }(x) + (1+x+x^{2})^{-2} \varphi(x) = 0 \Comma \quad x \in [0, 1] \Comma \\
		\varphi(0) = 1 \Comma \\
		\varepsilon\varphi^{\prime}(0) = 1 \Comma
	\end{cases}
\end{align}
where the exact solution $\varphi_{exact}$ is given by\footnote{We found the exact solution by using the Symbolic Math Toolbox of \MATLAB.}
\begin{align}\label{second_example_exact_sol}
    \varphi_{exact}(x)&= \frac{a(x)^{-1/4}}{\sqrt{3}\gamma(\varepsilon)}\sin\left(\gamma(\varepsilon)\left(\arctan\left(\frac{2x+1}{\sqrt{3}}\right)-\frac{\pi}{6}\right)\right)\nonumber\\
    &\quad-a(x)^{-1/4}\cos\left(\gamma(\varepsilon)\left(\arctan\left(\frac{2x+1}{\sqrt{3}}\right)-\frac{\pi}{6}\right)\right)\Comma
\end{align}
%
where $\gamma(\varepsilon):=\sqrt{3\varepsilon^{2}+4}/(\sqrt{3}\varepsilon)$.

This example is special in the sense that $a(x)=(1+x+x^{2})^{-2}$ belongs to the class of functions represented as $(C_{1}+C_{2}x+C_{3}x^{2})^{-2}$, with constants $C_{i}$, $i=1,2,3$, satisfying $|C_{2}|+|C_{3}|>0$ and $C_{2}^{2}\neq 4C_{1}C_{3}$. Further details regarding this class of functions are discussed in Appendix~\ref{appendix_convergent}, particularly with regard to the corresponding WKB series. Notably, for such functions it holds that $S_{1}' \not\equiv 0$, $S_{2}'\not\equiv 0$ and $S_{3}'\equiv 0$, see Remark~\ref{remark_a_C1C2C3}. Moreover, according to Proposition~\ref{proposition_catalan} and Remark~\ref{remark_catalan}, it follows that
\begin{align}
    S_{2n}'&=\mathcal{O}\left((n-1)^{-3/2}|C_{1}C_{3}-C_{2}^{2}/4|^{n-1}\right)\Comma \quad n\to \infty\Comma\\
    S_{2n+1}'&\equiv 0\Comma \quad n\geq 1\period
\end{align}
Consequently, this implies that the underlying asymptotic series (\ref{wkb_S}) is (geometrically) convergent for any $\varepsilon \leq |C_{1}C_{3}-C_{2}^{2}/4|^{-1/2}$, see again Remark~\ref{remark_catalan}. In this case, given that $|C_{1}C_{3}-C_{2}^{2}/4|=3/4$, the functions $S_{2n}'$ (and hence $S_{2n}$) exhibit exponential decay as $n\to\infty$, uniformly in $x\in I$. The corresponding WKB series is convergent for any $\varepsilon\in (0,2/\sqrt{3}]$. 

%
%
\begin{figure}
	\centering
    \includegraphics[scale=0.75]{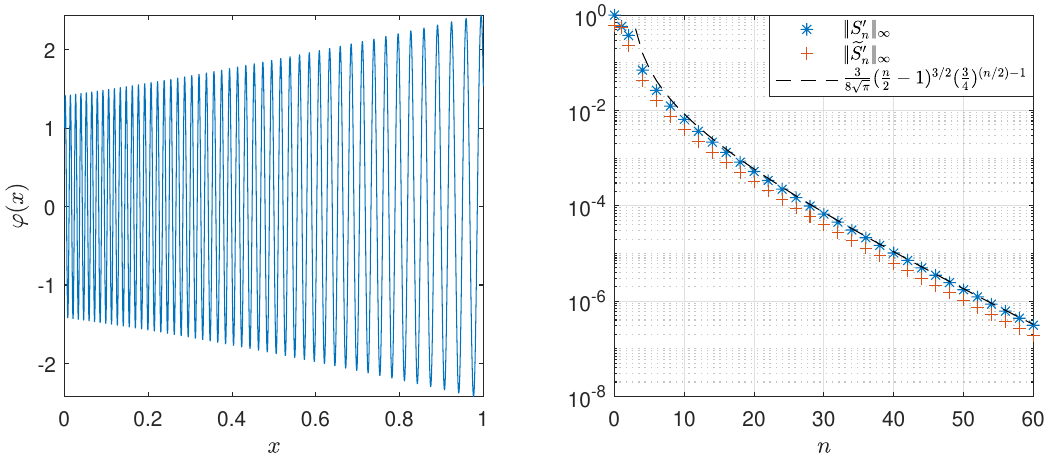}
 \caption{Left: Exact solution (\ref{second_example_exact_sol}) of IVP (\ref{eqn:second_example}) for the choice $\varepsilon=2^{-9}$. Right: $L^{\infty}(I)$-norm of $S_{n}'$ and $\widetilde{S}_{n}$ as functions of even $n$, for the example $a(x)=(1+x+x^{2})^{-2}$ on the interval $I=[0,1]$. The dashed line is proportional to the r.h.s.\ of (\ref{dS2n_asymptotics}) with $C_{1}=C_{2}=C_{3}=1$.}
	\label{plot:second_example_sol_plot}
\end{figure}
In Figure~\ref{plot:second_example_sol_plot} on the left we plot $\varphi_{exact}$ for the choice $\varepsilon=2^{-9}$.
%
%
Moreover, on the right of Figure~\ref{plot:second_example_sol_plot} we plot the $L^{\infty}(I)$-norm of $S_{n}'$ and $\widetilde{S}_{n}$, both as a function of $n$. Here, we set $M=30$ for the numerical integration of the functions $S_{n}^{\prime}$. We observe that the norms indeed decay exponentially, in agreement with Remark~\ref{remark_catalan}. Here, the dashed line is precisely given by the r.h.s.\ of (\ref{dS2n_asymptotics}) with $C_{1}=C_{2}=C_{3}=1$.
\begin{figure}
	\centering
    \includegraphics[scale=0.75]{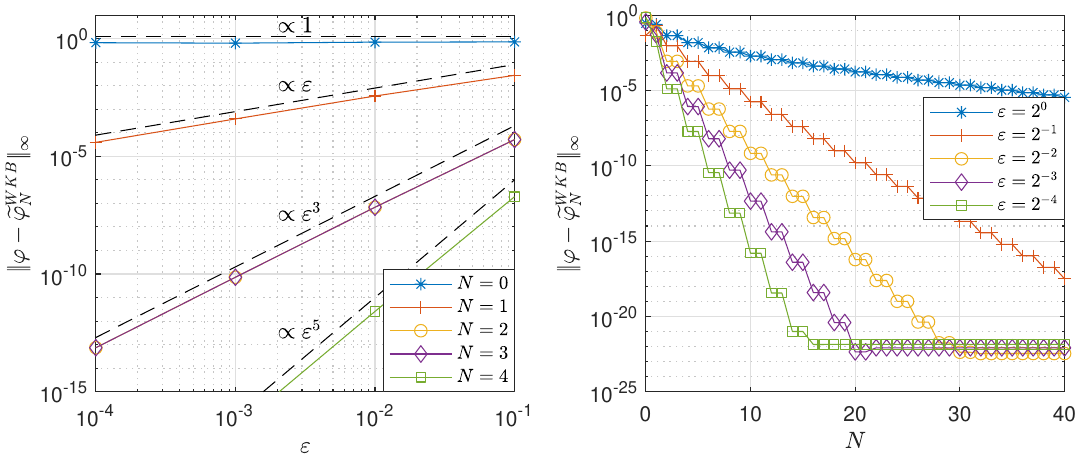}
	\caption{Left: $L^{\infty}(I)$-norm of the error of the WKB approximation as a function of $\varepsilon$, for the IVP (\ref{eqn:second_example}) and several choices of $N$. Here, we set $M=30$. The yellow curve for $N=2$ is the same as for $N=3$ and hence not visible in the shown plot. Right: $L^{\infty}(I)$-norm of the error of the WKB approximation as a function of $N$, for the IVP (\ref{eqn:second_example}) and several choices of $\varepsilon$. Here, we set $M=30$.
    }
	\label{plot:second_example_err_vs_eps}
\end{figure}
In Figure~\ref{plot:second_example_err_vs_eps} on the left we plot for $N=0,\dots,4$ the error of the WKB approximation $\lVert \varphi-\widetilde{\varphi}_{N}^{WKB}\rVert_{L^{\infty}(I)}$ as a function of $\varepsilon$. By comparing the results with (\ref{wkb_error_including_approx_order}), we observe that all $\mathcal{O}(\varepsilon^{n-1})e_{n}$-terms are essentially eliminated. Further, the error curves for $N=0,1,3$ behave like $\mathcal{O}(\varepsilon^{N})$ whereas the curves corresponding to the choices $N=2,4$ behave like $\mathcal{O}(\varepsilon^{N+1})$. This is because the given function $a$ implies $S_{2n+1}'\equiv 0$, for any $n\geq 1$, which means $\varphi_{N}^{WKB}=\varphi_{N+1}^{WKB}$ for any even $N\geq 2$.

Finally, on the right of Figure~\ref{plot:second_example_err_vs_eps} we plot the error $\lVert \varphi-\widetilde{\varphi}_{N}^{WKB}\rVert_{L^{\infty}(I)}$ as a function of the truncation order $N$, for several $\varepsilon$-values. We observe that all shown error curves are decreasing functions in $N$, up to the point where they reach values of approximately $10^{-22}$. This is due to the approximation of the functions $S_{n}$. More precisely, the first term of the sum in (\ref{wkb_error_including_approx_order}), namely, the $\mathcal{O}(\varepsilon^{-1})e_{0}$-term  corresponding to the approximation of $S_{0}$, becomes dominant at this point. For this reason, the minimum achievable error level is growing with decreasing $\varepsilon$. Besides from this saturation effect, the plot aligns well with Remark~\ref{remark_catalan}, suggesting that the WKB approximation converges to the exact solution of IVP (\ref{eqn:second_example}) as $N\to\infty$, for all displayed $\varepsilon$-values. Furthermore, one can observe again the fact that $\varphi_{N}^{WKB}=\varphi_{N+1}^{WKB}$ for even $N\geq 2$, as indicated by the step-like behavior of all shown error curves.
%
%
%
\section{Conclusion}\label{section_conclusion}
In the present paper we analyzed the WKB approximation of the solution to a highly oscillatory initial value problem. Assuming that the potential in the equation is analytic, we found explicit upper bounds for the terms occurring in the asymptotic WKB series of the approximate solution. Building on that, we proved error estimates which are explicit not only w.r.t.\ the small parameter $\varepsilon$ but also w.r.t.\ $N$, the chosen number of terms in the truncated asymptotic series. We showed that the optimal truncation order $N_{opt}$ is proportional to $\varepsilon^{-1}$, and this results in an approximation error that is exponentially small w.r.t.\ $\varepsilon$. We confirmed our theoretical results by several numerical experiments.
\section*{Declaration of competing interest}
The authors declare that they have no known competing financial interests or personal relationships that could have appeared to influence the work reported in this paper.
\section*{Acknowledgements}
The authors J. Körner, A. Arnold, and J. M. Melenk acknowledge support by the Austrian Science Fund (FWF) project \href{https://doi.org/10.55776/F65}{10.55776/F65}, and the authors A. Arnold and C. Klein also by the bi-national FWF-project I3538-N32. C. Klein thanks for support by  the ANR-17-EURE-0002 EIPHI and by the 
European Union Horizon 2020 research and innovation program under the 
Marie Sklodowska-Curie RISE 2017 grant agreement no. 778010 IPaDEGAN. 
\appendix
\section{Convergent WKB series}\label{appendix_convergent}
In this appendix we provide examples where the asymptotic series (\ref{wkb_S}) is convergent in $L^{\infty}(I)$.
In practice, the norms $\lVert S_{n}'\rVert_{L^{\infty}(I)}$ (and $\lVert S_{n}\rVert_{L^{\infty}(I)}$) often decrease up to a certain number of $n$ before they start to increase rapidly, e.g., see the right plot of Figure~\ref{plot:airy_sol_plot}. However, there are examples where one can easily verify that this is not the case. For instance, consider the simplest case in which $a\equiv a_{0}$ is constant. By (\ref{S1}) this is equivalent to $S_{1}^{\prime}\equiv 0$, which by (\ref{Sn}) then implies $S_{n}^{\prime}\equiv 0$ for every $n\geq 1$. Similarly, one easily verifies that $S_{2}'\equiv 0$ is equivalent to $a$ having the form $a(x)=(C_{1}+C_{2}x)^{-4}$ for some constants $C_{1}$ and $C_{2}$, see also \cite[Problem 10.2]{Bender1999AdvancedMM}. It then holds $S_{n}^{\prime}\equiv 0$ for every $n\geq 2$. Thus, in both of the just mentioned cases, the asymptotic series (\ref{wkb_S}) terminates automatically and is therefore convergent. The corresponding WKB approximation (\ref{general_wkb_solution}) with $N\geq 0$ (respectively $N\geq 1$) is then the exact solution to IVP (\ref{schroedinger_eq}). Indeed, revisiting (\ref{residual_f_eps}), it is clear that the r.h.s.\ in (\ref{wkb_error_1}) then vanishes, i.e.\ $\lVert \varphi-\varphi_{N}^{WKB}\rVert_{L^{\infty}(I)}=0$.

In the subsequent discussion, we will give examples of convergent WKB series which do not terminate automatically.
\begin{Proposition}\label{proposition_catalan}
	Let $S_{3}^{\prime}\equiv 0$. Then it holds
	\begin{align}
		S_{2n}^{\prime}&=S_{2}^{\prime}\left(-\frac{S_{2}'}{2S_{0}^{\prime}}\right)^{n-1}a_{n}\Comma\label{S_n_even}\\
		S_{2n+1}^{\prime}&\equiv 0\Comma\label{S_n_odd}
	\end{align}
	for $n\geq 2$. Here, the sequence $a_{n}$ is recursively defined by $a_{1}:=1$ and
	\begin{align}
		a_{n+1}:=\sum_{j=1}^{n}a_{j}a_{n+1-j}\Comma\quad n\geq 1\period
	\end{align}
\end{Proposition}
\begin{proof}
	It is easy to check, that (\ref{S_n_even}) and (\ref{S_n_odd}) hold for $n=2$. We proceed now by induction on $n$. To this end, assume that formulas (\ref{S_n_even}) and (\ref{S_n_odd}) hold for all $2\leq k\leq n$ for some fixed $n\geq 2$. We shall now prove them for $n+1$. The induction hypothesis implies that $S_{2n+1}^{\prime\prime}\equiv 0$ as well as $S_{j}^{\prime}\equiv 0$ for all odd indices $j$ such that $1\leq j\leq 2n+1$. Hence,
	\begin{align}
		S_{2n+2}^{\prime}&=-\frac{1}{2S_{0}^{\prime}}\left(\sum_{j=1}^{2n+1}S_{j}^{\prime}S_{2n+2-j}^{\prime}+S_{2n+1}^{\prime\prime}\right)\nonumber\\
		&=-\frac{1}{2S_{0}^{\prime}}\sum_{j=1}^{n}S_{2j}^{\prime}S_{2(n+1-j)}^{\prime}\nonumber\\
		&=S_{2}^{\prime}\left(-\frac{S_{2}'}{2S_{0}^{\prime}}\right)^{n}\sum_{j=1}^{n}a_{j}a_{n+1-j}\nonumber\\
		&=S_{2}^{\prime}\left(-\frac{S_{2}'}{2S_{0}^{\prime}}\right)^{n}a_{n+1}\Comma\label{S_n_odd_proof}
	\end{align}
	where we have again used the induction hypothesis in the third equation. Differentiating (\ref{S_n_odd_proof}) and using $\frac{S_{0}^{\prime\prime}}{S_{0}^{\prime}}=-2S_{1}^{\prime}$ we further obtain
	\begin{align}
		S_{2n+2}^{\prime\prime}=\left(-\frac{S_{2}'}{2S_{0}^{\prime}}\right)^{n}\Big(2nS_{1}^{\prime}S_{2}^{\prime}+(n+1)S_{2}^{\prime\prime}\Big)a_{n+1}\period
	\end{align}
	Moreover, the induction hypothesis implies $S_{j}^{\prime}S_{2n+3-j}^{\prime}\equiv 0$ for $2\leq j\leq 2n+1$ since either $j$ or $2n+3-j$ is odd. Therefore, we get
	\begin{align}
		S_{2n+3}^{\prime}&=-\frac{1}{2S_{0}^{\prime}}\left(\sum_{j=1}^{2n+2}S_{j}^{\prime}S_{2n+3-j}^{\prime}+S_{2n+2}^{\prime\prime}\right)\nonumber\\
		&=-\frac{1}{2S_{0}^{\prime}}\left(2S_{1}'S_{2n+2}'+S_{2n+2}^{\prime\prime}\right)\nonumber\\
		&=\left(-\frac{S_{2}'}{2S_{0}^{\prime}}\right)^{n}\left(-\frac{1}{2S_{0}'}\right)\Big(2(n+1)S_{1}'S_{2}'+(n+1)S_{2}''\Big)a_{n+1}\nonumber\\
		&=\left(-\frac{S_{2}'}{2S_{0}^{\prime}}\right)^{n}(n+1)S_{3}^{\prime}a_{n+1}\nonumber\\
		&\equiv0\Comma
	\end{align}
by assumption on $S_{3}'$. This concludes the proof.
\end{proof}
\begin{Remark}\label{remark_a_C1C2C3}
	Proposition~\ref{proposition_catalan} assumes $S_{3}^{\prime}\equiv 0$, which is equivalent to $a(x)$ satisfying the third order nonlinear ODE $15a'^3 + 4a^2a^{\prime\prime\prime} - 18aa'a''=0$. With the aid of \MATLAB's Symbolic Math Toolbox we find that the general solution to this ODE is given by $a(x)=(C_{1}+C_{2}x+C_{3}x^2)^{-2}$, where $C_{1},C_{2}$ and $C_{3}$ are constants. A simple computation then shows that if $|C_{2}|+|C_{3}|>0$ and $C_{2}^{2}\neq 4C_{1}C_{3}$, the coefficient function $a$ does not have one of the two forms mentioned before Proposition~\ref{proposition_catalan}, i.e., $S_{1}'\not\equiv 0$ and $S_{2}'\not\equiv 0$. Thus, due to Proposition~\ref{proposition_catalan}, the corresponding WKB series does not terminate in this case.
\end{Remark}
\begin{Remark}\label{remark_catalan}
	The numbers $a_{n}=:c_{n-1}$ in Proposition~\ref{proposition_catalan} are the so-called Catalan numbers (e.g., see \cite{Graham1991ConcreteM}), which are known to grow asymptotically as $c_{n}\sim\frac{4^{n}}{n^{3/2}\sqrt{\pi}}$, for $n\to\infty$. Let us assume that $a(x)=(C_{1}+C_{2}x+C_{3}x^2)^{-2}$ such that $S_{3}'\equiv 0$, see Remark~\ref{remark_a_C1C2C3}. We then have $S_{2}'(x)/S_{0}'(x)=C_{1}C_{3}/2 - C_{2}^2/8$. According to (\ref{S_n_even}), we thus have for $n\geq 2$
    \begin{align}\label{dS2n_asymptotics}
        \lVert S_{2n}^{\prime}\rVert_{L^{\infty}(I)}&\leq\lVert S_{2}^{\prime}\rVert_{L^{\infty}(I)}\left\lVert\frac{S_{2}^{\prime}}{2S_{0}^{\prime}}\right\rVert_{L^{\infty}(I)}^{n-1}c_{n-1}\nonumber\\
        &\sim\frac{\lVert S_{2}^{\prime}\rVert_{L^{\infty}(I)}}{(n-1)^{3/2}\sqrt{\pi}}\left\lVert\frac{2S_{2}^{\prime}}{S_{0}^{\prime}}\right\rVert_{L^{\infty}(I)}^{n-1}\Comma\quad n\to\infty\nonumber\\
		&=\frac{\lVert S_{2}^{\prime}\rVert_{L^{\infty}(I)}}{(n-1)^{3/2}\sqrt{\pi}}\left|C_{1}C_{3}-\frac{C_{2}^{2}}{4}\right|^{n-1}\period
    \end{align}
    By definition (\ref{Sn_definition}), we conclude that
    \begin{align}\label{S_2n_asymptotic}
        \lVert S_{2n}\rVert_{L^{\infty}(I)}=\mathcal{O}\left((n-1)^{-3/2}|C_{1}C_{3}-C_{2}^{2}/4|^{n-1}\right)\Comma \quad n\to \infty\period
    \end{align}
    Thus, the constants $C_{i}$, $i=1,2,3$, determine whether the function $\lVert S_{2n}\rVert_{L^{\infty}(I)}$ is exponentially growing or decaying as $n\to\infty$. Note that Proposition~\ref{proposition_catalan} also implies that $S_{2n+1}\equiv 0$ for $n\geq 1$. A short calculation then shows that (\ref{S_2n_asymptotic}) implies that the corresponding WKB series $\exp(\sum_{n=0}^{\infty}\varepsilon^{n-1}S_{n}(x))$ is (geometrically) convergent for any $x\in I$, if $\varepsilon \leq |C_{1}C_{3}-C_{2}^{2}/4|^{-1/2}$.
\end{Remark}
%
	
   

\bibliographystyle{abbrv}
\bibliography{references}


\end{document}